\documentclass[english,12pt]{amsart}
\usepackage[utf8]{inputenc}
\usepackage[T1]{fontenc}
\usepackage{ucs}
\usepackage[top=1in, bottom=1in, left=1in, right=1in]{geometry}
\usepackage{color}
\usepackage{amsmath}
\usepackage{amsfonts,dsfont}
\usepackage[psamsfonts]{amssymb}
\usepackage{mathtools} 
\usepackage{amsthm}
\usepackage{mathrsfs} 
\usepackage{graphicx}
\usepackage{enumitem}  
\usepackage{hyperref,soul}

\usepackage{xcolor}

\newtheorem{theoA}{Theorem}

\newtheorem{propA}{Proposition}

\renewcommand{\theequation}{\thesection.\arabic{equation}}
\newtheorem{theorem}{Theorem}
\newtheorem{lemma}{Lemma}
\newtheorem{proposition}{Proposition}
\newtheorem{corollary}{Corollary}
\newtheorem{remark}{Remark}
\newtheorem{definition}{Definition}

\newtheorem{notation}{Notations}

\usepackage{thmtools}
\declaretheoremstyle[notefont=\bfseries,notebraces={}{},%
headpunct={},postheadspace=1em]{mystyle}
\declaretheorem[style=mystyle,numbered=no,name=Theorem]{thm-hand}
\newcommand{\eqnsection}{
\renewcommand{\theequation}{\thesection.\arabic{equation}}
    \makeatletter
    \csname  @addtoreset\endcsname{equation}{section}
    \makeatother}
\eqnsection

\def\d{\, \mathrm{d}}

\def\ali{\hfill\break}
\def\w{\omega}

\def\demi{{1\over 2}}
\def\ddd{{\cal D}}

\def\Z{{{\Bbb Z}}}
\def\P{{{\Bbb P}}}
\def\N{{{\Bbb N}}}
\def\R{{{\Bbb R}}}

\def\ppp{{{\mathcal P}}}
\def\aaa{{{\mathcal A}}}

\def\fff{{{\mathcal F}}}

\def\ttt{{\mathcal T}}

\def\E{{{\Bbb E}}}

\def\Id{{\hbox{Id}}}
\def\bbb{{{\mathcal B}}}

\def\ggg{{\mathcal{G}}}

\def\simGun{{\stackrel{\ggg_1}{\sim}}}
\def\ddd{{\mathcal{D}}}
\def\skorokhod{{D}}

\usepackage{verbatim}

\DeclareFontFamily{U}{mathx}{\hyphenchar\font45}
\DeclareFontShape{U}{mathx}{m}{n}{
      <5> <6> <7> <8> <9> <10>
      <10.95> <12> <14.4> <17.28> <20.74> <24.88>
      mathx10
      }{}
\DeclareSymbolFont{mathx}{U}{mathx}{m}{n}
\DeclareFontSubstitution{U}{mathx}{m}{n}
\DeclareMathAccent{\widecheck}{0}{mathx}{"71}
\DeclareMathAccent{\wideparen}{0}{mathx}{"75}

\def\hat{\widehat}
\def\tilde{\widetilde}
\def\check{\widecheck}

\graphicspath{{./images/} }  
\begin{document}
\author[C. Sabot]{Christophe SABOT}
\address{Universit\'e de Lyon, Universit\'e Lyon 1,
Institut Camille Jordan, CNRS UMR 5208, 43, Boulevard du 11 novembre 1918,
69622 Villeurbanne Cedex, France} \email{sabot@math.univ-lyon1.fr}
\author[X. Zeng]{Xiaolin Zeng}\thanks{This work was supported by the LABEX MILYON (ANR-10-LABX-0070) of Universit\'e de Lyon, within the program "Investissements d'Avenir" (ANR-11-IDEX-0007) operated by the French National Research Agency (ANR), and by the ANR/FNS project MALIN (ANR-16-CE93-0003). The second author is supported
  by ERC Starting Grant 678520.}
\address{108 Schreiber building,
School of mathematics,
Tel aviv university, P.O.B. 39040,
Ramat aviv, Tel aviv 69978, Israel} 
\email{xzeng@math.univ-lyon1.fr}
\title[A random Schr\"odinger operator associated with the VRJP on infinite graphs]{A
random Schr\"odinger operator associated with the Vertex Reinforced Jump Process on infinite graphs}
\begin{abstract}
This paper concerns the Vertex reinforced jump process (VRJP), the Edge reinforced random walk (ERRW) and their link with a random Schr\"odinger operator.
On infinite graphs, we define a 1-dependent random potential \(\beta\) extending that defined in~\cite{sabot2011edge} on finite graphs,
and consider its associated random Schr\"odinger operator \(H_\beta\).
We construct a random function \(\psi\) as a limit of martingales, such that \(\psi=0\) when the VRJP is recurrent, and \(\psi\) is a positive generalized eigenfunction of the random Schr\"odinger operator with eigenvalue \(0\), when the VRJP is transient.
Then we prove a representation of the VRJP on infinite graphs as a mixture of Markov jump processes involving the function \(\psi\), the Green function of the random Schr\"odinger operator
and an independent Gamma random variable.
On \(\Z^d\), we deduce from this representation a zero-one law for recurrence or transience of the VRJP and the ERRW, and a functional central limit theorem for the VRJP and the ERRW at weak reinforcement in dimension \(d\ge 3\), using estimates of
  \cite{disertori2010quasi,disertori2014transience}.
Finally, we deduce recurrence of the ERRW in dimension \(d=2\) for any initial constant weights (using the estimates of Merkl and Rolles, \cite{merklbounding,merkl2009recurrence}),
thus giving a full answer to the old question of Diaconis.
  We also raise some questions on the links between recurrence/transience of
  the VRJP and localization/delocalization of the random Schr\"odinger operator \(H_\beta\).

\end{abstract}

\maketitle

\section{Introduction}
This paper concerns the Vertex Reinforced Jump Process (VRJP) and the Edge Reinforced Random Walk (ERRW) and their relation with
a random Schr\"odinger operator associated with a stationary 1-dependent random potential
(i.e.\ the potential is independent at distance larger or equal to 2).

The VRJP is a continuous time self-interacting process introduced in \cite{davis2004vertex}, investigated on trees in \cite{collevecchio2009limit,basdevant2012continuous}
and on general graphs in \cite{sabot2011edge,sabot2013ray}.
We first recall its definition.
Let \(\ggg=(V,E)\) be an undirected graph with finite degree at each vertex.
We write \(i\sim j\) if \(i\in V\), \(j\in V\) and \(\{i,j\}\) is an edge of the graph. We always assume that the graph is connected
and has no trivial loops (i.e.\ vertex \(i\) such that \(i\sim i\)).
Let \((W_{i,j})_{i\sim j}\) be a set of positive conductances, \(W_{i,j}>0\), \(W_{i,j}=W_{j,i}\).
The VRJP is the continuous time process \((Y_s)_{s\ge0}\) on \(V\), starting at time \(0\) at some vertex \(i_0\in V\),
which, conditionally on the past at time \(s\), if \(Y_s=i\),   jumps to a neighbour \(j\) of \(i\) at rate
\[W_{i,j}L_j(s),\]
where
\[L_j(s):=1+\int_0^s \mathds{1}_{\{Y_u=j\}}\,du.\]
In~\cite{sabot2011edge}, Sabot and Tarr\`es introduced the following time change of the VRJP
\begin{equation}
\label{Chgt-time}
Z_t=Y_{D^{-1}(t)},
\end{equation}
where \(D(s)\) is the following increasing function
\begin{equation*}
D(s)=\sum_{i\in V} (L^2_i(s)-1).
\end{equation*}
We call this process the VRJP in exchangeable time scale and denote by \(\P_{i_{0}}^{\text{VRJP}}\) its law
starting from the vertex \(i_0\).
When the graph is finite it is proved in \cite[Theorem~2]{sabot2011edge} that the VRJP in exchangeable time scale \((Z_t)_{t\ge 0}\)  is a mixture of Markov jump processes.
More precisely, there exists a random field \((u_{j})_{j\in V}\) such that \(Z\) is a mixture of Markov jump processes
with jump rates from \(i\) to \(j\)
\[\demi W_{i,j}e^{u_{j}-u_{i}}.\]
The law of \((u_j)\) is explicit, c.f. \cite[Theorem~2]{sabot2011edge}, and forthcoming Theorem~\ref{thm_vrjp_u_ST}.
It appears to be a marginal of a supersymmetric sigma-field which had been investigated previously by
Disertori, Spencer, Zirnbauer (c.f. \cite{disertori2010anderson}, \cite{disertori2010quasi}, \cite{zirnbauer1991fourier}).
As a consequence of this representation and of~\cite{disertori2010anderson}, \cite{disertori2010quasi}, it was proved in~\cite{sabot2011edge} the following:
when the graph has bounded degree, there exists a real \(\lambda_0>0\) such that if \(W_{i,j}\le \lambda_0\) for all $i\sim j$, then the VRJP is positively recurrent,
more precisely, \(Z\) is a mixture of positive recurrent Markov jump processes. When the graph is the grid \(\Z^d\), with \(d\ge 3\), there exists
\(\lambda_1<+\infty\) such that if \(W_{i,j}\ge \lambda_1\) for all $i\sim j$, the VRJP is transient.  Hence, it shows a phase transition between recurrence
and transience in dimension \(d\ge 3\).
The question of the representation of the VRJP on infinite graphs as a mixture of Markov jump processes is non trivial, especially in the transient case.
It is possible to prove such a representation by a weak convergence argument, following~\cite{merkl2007random}, but it gives little information on the mixing law.
In this paper we prove such a representation involving the Green function and a generalized eigenfunction of a random Schr\"odinger operator.

Let us give a flavour of the main results of the paper in the case of the VRJP on \(\Z^d\) with \(W_{i,j}=W\) constant.
We construct a positive 1-dependent random potential \((\beta_j)_{j\in \Z^d}\) (i.e.\ two subsets of the \(\beta\)'s are independent if their indices are at least at distance 2)
and with marginal given by inverse of Inverse Gaussian law with parameters \(({1\over dW},1)\). This field is a natural
extension to infinite graphs of the field defined by Sabot, Tarr\`es, Zeng in~\cite{STZ}\@. We consider the random
Schr\"odinger operator
\[
H_\beta= -W\Delta +V,
\]
where \(\Delta\) is the usual discrete (non-positive) Laplacian and \(V\) is the multiplication operator defined by \(V_j= 2\beta_j -2dW\).
Hence, it corresponds to the Anderson model with a random potential which is not i.i.d.\ but only stationary and 1-dependent.
When the VRJP is transient we prove that there exists a positive generalized eigenfunction \(\psi\)
of \(H_\beta\) with eigenvalue 0, stationary and ergodic. Let
\((G(i,j))_{i\in Z^d, j\in \Z^d}\) be defined by
\[
G(i,j)= \hat G(i,j)+\frac{1}{2}\gamma^{-1} \psi(i)\psi(j),
\]
where \(\hat G=(H_\beta)^{-1}\) is the Green function (which happens to be well-defined in an appropriate sense)
and \(\gamma\) is an extra random variable independent of the field
\(\beta\) with law \(Gamma(\demi, 1)\). We prove the  following representation for the VRJP: the VRJP in exchangeable time scale \(Z\) starting from
the point \(i_0\) is a mixture of Markov jump processes with jump rates from \(i\) to \(j\)
\begin{eqnarray}\label{rep}
\demi W_{i,j} {G(i_0,j)\over  G(i_0,i)}.
\end{eqnarray}
When the VRJP is recurrent the same representation is valid with \(\psi=0\). In fact, the function \(\psi\) is the a.s.\ limit of
a martingale, the limit being positive when the VRJP is transient and 0 when the VRJP is recurrent.
It is remarkable that when the VRJP is recurrent it can be represented as a mixture with \(\beta\)-measurable jump rates, but when the VRJP
is transient it involves an extra independent Gamma random variable.
This representation extends to infinite graphs the representation given in \cite{STZ} for finite graphs. A new feature appears in the transient case,
where the generalized eigenfunction \(\psi\) is involved in the representation.
We suspect that recurrence/transience of the VRJP is related to localization/delocalization of the
random Schr\"odinger operator \(H_\beta\) at the bottom of the spectrum.

The representation (\ref{rep}) has several consequences on the VRJP and the ERRW.
The ERRW is a reinforced process introduced by Diaconis and Coppersmith in 1986 (see Section~\ref{results-ERRW} for a definition).
The recurrence of the 2-dimensional ERRW is a famous open question raised by Diaconis, see \cite{coppersmith1987random,pemantle1988phase,rolles2000edge,merkl2009recurrence} for early references.
Important progress has been done recently in the understanding of this process. In particular, in \cite{sabot2011edge},
an explicit relation between the ERRW and the VRJP was stated, thus somehow reducing the analysis of the ERRW to that of the VRJP.
In \cite{sabot2011edge,angel2014localization}, it was proved by rather different methods that the ERRW on any graph with bounded degree
at strong enough reinforcement is positive recurrent. In \cite{disertori2014transience}, it was proved that the ERRW is transient on \(\Z^d\), \(d\ge 3\),
at weak reinforcement.

The representation (\ref{rep}) allows us to complete the picture both in dimension 2 and in the transient regime.
More precisely, we prove a functional central limit theorem for the ERRW and for the discrete time process associated with the VRJP in dimension \(d\ge 3\) at weak
reinforcement, using the estimates of \cite{disertori2010quasi,disertori2014transience}.
 Using the polynomial estimate provided by Merkl and Rolles, \cite{merkl2009recurrence},
 we are able to prove recurrence of ERRW on \(\Z^2\) for all initial constant weights, hence giving a full answer to the question of Diaconis.
\section{Statements of the results}
\label{Resultats}
\subsection{ Notations}
\label{sec_Notations}
We denote by $\R_+$ (resp. $\R^*_+$) the set of non-negative (resp. positive) reals.

Let \(\mathcal{G}=(V,E)\) be an undirected, locally finite, connected graph without trivial loops or multiples edges. For \(i,j\in V\), write \(i\sim j\) if \(i\) is a neighbor of \(j\). We write \(\operatorname{d}_{\mathcal{G}}\) for the graph distance in \(\mathcal{G}\), and for two subsets \(U,U’\) of \(V\), we define \(\operatorname{d}_{\mathcal{G}}(U,U’)=\inf_{i\in U,j\in U’}\operatorname{d}_{\mathcal{G}}(i,j)\). We suppose given, for each edge \(e=\{i,j\}\in E\), a positive real \(W_{i,j}>0\), understood as the conductance of the edge \(e\). In this case we call $(\ggg, (W_e)_{e\in E})$ a graph with conductances.

\noindent{\underline {Convention}:} We adopt the notation \(\sum_{i\sim j} \) for the sum on all undirected edges \(\{i,j\}\), counting only once each edge.

When $\beta=(\beta_i)_{i\in V}\in \R^V$ is a real vector indexed by the vertices and $U\subset V$,
we write $\beta_U$ for the restriction of $\beta$ to $U$, i.e. $\beta_U=(\beta_i)_{i\in U}$. When $A=(A_{i,j})_{i,j\in V}\in \R^{V\times V}$ is a real function on $V\times V$ and $U\subset V$, $U'\subset V$, we write $A_{U,U'}$ for the restriction of $A$ to $U\times U'$, i.e. $A_{ U,U'}=(A_{i,j})_{i\in U,j\in U'}$.  

It will be convenient to define the continuous time processes that appear in the text on the same canonical space. In the sequel, we will denote by $\skorokhod([0,\infty),V)$ the space of c\`adl\`ag
functions from $[0,\infty)$ to $V$.  The law of the VRJP in exchangeable time scale defined in \eqref{Chgt-time}, starting from $i_0$, will be denoted by \(\P_{i_{0}}^{\text{VRJP}}\), which is a probability on $\skorokhod([0,\infty),V)$.
The VRJP will always be defined on the canonical space and $(Z_t)_{t\in \R_+}$ will denote the canonical process defined by $Z_t(\w)=\w(t)$ for $\w\in \skorokhod([0,\infty),V)$.
\begin{remark}
\label{rk_mult_edge}
We do not allow multiple edges or trivial loops since it does not bring more generality to the VRJP. Indeed, from its definition, it follows that the VRJP on a graph with multiple edges and trivial loops has the same law as the VRJP on the graph where trivial loops are removed and multiple edges are replaced by a single edge by summing the conductances of the multiples edges. Similarly, the law on random potentials that appears in the sequel can always be reduced to graphs without multiple edges or trivial
loops. Nevertheless, in Section~\ref{ss_martingale} it simplifies notations to allow trivial loops.
\end{remark}
\subsection{Representation of the VRJP on infinite graphs}
\label{ss_main-results}
Define the operator \(P=(P_{i,j})_{i,j\in V}\) by
\[
P_{i,j}=\left\{\begin{array}{ll}
W_{i,j},& \hbox{ if \(i\sim j\)},
\\0, & \hbox{ otherwise}.
\end{array}
\right.
\]
We define below a probability distribution on potentials on the graph. A potential on the graph will generically be denoted $\beta=(\beta_i)_{i\in V}\in \R^V$. 
With the potential $\beta\in \R^V$, we associate the Schr\"odinger operator on \(\mathcal{G}\)
\begin{eqnarray}\label{def_Hbeta}
H_\beta = -P+2\beta,
\end{eqnarray}
where \(\beta\) represents the operator of multiplication by the potential \((\beta_i)\) (or equivalently the diagonal operator with diagonal terms $(\beta_i)_{i\in V}$).

We denote by
\begin{align}
\label{DVW}
\ddd_V^W=\{\beta\in\R^V, \;\; (H_\beta)_{U,U}>0\hbox{ for all finite subsets $U\subset V$}\},
\end{align}
where $(H_\beta)_{U,U}>0$ means that the restriction of $H_\beta$ to $U\times U$ is positive definite. Obviously, $\ddd_V^W\subset (\R_+^*)^V$ since when $U=\{i\}$ the restriction of $H_\beta$ is the real $2\beta_i$. We endow $\ddd_V^W$ with its Borelian $\sigma$-field denoted $\bbb(\ddd_V^W)$.

The following statement extends the random potential defined in \cite[Theorem~1]{STZ} to infinite graphs.
\begin{proposition}
\label{kolmogorov} Let $(\ggg,(W_e)_{e\in E})$ be a graph with conductances as defined in Section~\ref{sec_Notations}.
There exists a unique probability distribution $\nu_{V}^W$ defined on 
$(\ddd_V^W,\bbb(\ddd_V^W))$,
such that for any finite subset \(U\subset V\)
and any \((\lambda_i)_{i\in U}\in \R_{+}^U\):
\begin{eqnarray*}
\int e^{-\sum_{i\in U} \lambda_i \beta_i} \nu_V^W(d\beta)=e^{-\sum_{i\sim j, \; i,j\in U} W_{i,j}(\sqrt{(1+\lambda_i)(1+\lambda_j)}-1)
-\sum_{i\sim j, i \in U, j\notin U} W_{i,j}(\sqrt{1+\lambda_i}-1)}{1\over \prod_{i\in U} \sqrt{1+\lambda_i}}.
\end{eqnarray*}
In particular, we have the following properties: on the probability space $(\ddd_V^W,\bbb(\ddd_V^W),\nu_V^W(d\beta)$),
\begin{itemize}
\item
{\it(1-dependence)} if \(U,U’\subset V\) are such that \(\operatorname{d}_{\mathcal{G}}(U, U’)\geq 2\), then the random variables \(\beta\mapsto \beta_U\) and \(\beta\mapsto\beta_{U’}\)
are independent,
\item
{\it(Reciprocal inverse Gaussian marginals)}
for $i\in V$, the random variable \(\beta\mapsto \frac{1}{2\beta_i}\) has an inverse Gaussian distribution with parameter \((\frac{1}{W_{i}},1)\) where \(W_i=\sum_{j\sim i} W_{i,j}\).
\end{itemize}
\end{proposition}
\begin{remark}
On finite graphs, the density of $\nu_V^W$ is explicit,
c.f.~\cite[Theorem 1]{STZ} and Theorem~\ref{thm_potential_fini} below.
\end{remark}

In the sequel, the probability space 
$(\ddd_V^W,\bbb(\ddd_V^W),\nu_V^W)$
will be considered as the canonical space of random potentials on the graph. We write $\E_{\nu_V^W}$ for the expectation with respect to $\nu_V^W$.
We will introduce several random variables on this probability space, and adopt the following notation: when $\beta\mapsto X_\beta$ is a measurable function we will write $X$ for the associated random variable and $X_\beta$ for its realization on the potential $\beta$. In particular, we will write $H$ for the random Schr\"odinger operator $\beta\mapsto H_\beta$ defined above. By abuse of notation, we sometimes consider $\beta_i$ for $i\in V$ or $\beta_U$ for $U\subset V$ as random variables (more precisely, the random variables are $\beta\mapsto \beta_i$ and $\beta\mapsto \beta_U$).

\begin{definition}
\label{def-G-and-psi}
Let \((V_{n})_{n\in \N}\) be an increasing  sequence of finite connected subsets of \(V\) such that
\[
\cup_{n=0}^\infty V_n =V.\]
For $n\in \N$, we define $\fff^{(n)}\subset \bbb(\ddd_V^W)$ as the sub $\sigma$-field generated by the random variable $\beta\mapsto \beta_{V_n}$. 
For $n\in \N$ and $\beta\in \ddd_V^W$, we define a random operator $(\hat G_\beta^{(n)}(i,j))_{i,j\in V}$ by
\[
\hat G_\beta^{(n)}(i,j) =
\begin{cases}
((H_\beta)_{V_n,V_n})^{-1}(i,j) , & \hbox{ if \( i,j \in V_n\),}
\\
0, & \hbox{ otherwise.}
\end{cases}
\]
For $n\in \N$ and $\beta\in \ddd_V^W$, we define a random function $(\psi_\beta^{(n)}(i) )_{i\in V}$
as the unique solution of the following equation:
$$
\begin{cases}
H_\beta (\psi_\beta^{(n)})(i)=0, &\hbox{ for $i\in V_n$,}
\\
\psi_\beta^{(n)}(i)=1, &\hbox{ for $i\in V^c_n$.}
\end{cases}
$$
By definition, the random variables $\hat G^{(n)}:\beta\mapsto \hat G^{(n)}_\beta$ and $\psi^{(n)}:\beta\mapsto\psi^{(n)}_\beta$  are $\fff^{(n)}$-measurable.
\end{definition}
The fact that there is a unique solution to the equation defining $\psi^{(n)}_\beta$ is elementary, see the proof in Section~\ref{ss_Kolmogorov}.

Our main theorem is the following.
\begin{theorem}
\label{convergence}
\begin{enumerate}[label=(\roman*),leftmargin=*]
\item
\label{convergence_i}
For all $i,j\in V$, the sequence of random variables \(\hat{G}^{(n)}(i,j)\) is non-decreasing and converges a.s. to 
\[
\hat G(i,j):= \lim_{n\to \infty} \hat G^{(n)}(i,j).
\]
Moreover, $\nu_V^W$-almost surely, $0<\hat G(i,j)<\infty$ and the limit does not depend on the choice of the sequence of subsets $V_n$. 
\item
\label{convergence_ii}
Under the probability $\nu_V^W$, for all \(i\in V\), \(\psi^{(n)}(i)\) is a positive \(\fff^{(n)}\)-martingale. It converges a.s.\ to a
random variable \(\psi(i)\), such that $\psi(i)\ge 0$ a.s., and the limit does not depend on the choice of the
increasing sequence \((V_n)\).
Moreover, the quadratic variation of the vectorial martingale \((\psi^{(n)}(i))_{i\in V}\) is given a.s. by
\[\left<\psi(i),\psi(j)\right>_{n}=\hat G^{(n)}(i,j).
\]
In particular,
\(\psi^{(n)}(i)\) is bounded in \(L^2\) if and only if \(\E_{\nu_V^W}(\hat G(i,i))<\infty\).

\item
\label{convergence_iii}
For any real $\gamma>0$ and $\beta\in \ddd_V^W$, we define
\[
G_{\beta,\gamma}(i,j)= \hat G_\beta(i,j)+\frac{1}{2}\gamma^{-1} \psi_\beta(i)\psi_\beta(j).
\]
For $i_0\in V$ and $x\in V$, denote by \(P_x^{\beta,\gamma,i_{0}}\) the law of the Markov jump process which starts at $x\in V$ and jumps from $i$ to $j$ at rate
\begin{eqnarray}\label{jumping-G}
\frac{1}{2}W_{i,j} \frac{G_{\beta,\gamma} (i_{0},j)}{G_{\beta,\gamma} (i_{0},i)}.
\end{eqnarray}
Then the VRJP in exchangeable time scale, defined in section \ref{sec_Notations}, with conductances $(W_{i,j})$ and starting from $i_0$ is a mixture of these Markov jump processes and has law
\begin{eqnarray}\label{rep_VRJP}
\P^{\text{VRJP}}_{i_{0}}(\ \cdot\ )=\int P_{i_0}^{\beta,\gamma,i_{0}}(\ \cdot \ ) \nu_V^{W}(d\beta) \frac{\mathds{1}_{\gamma>0}}{\sqrt{\pi\gamma}}e^{-\gamma}d\gamma.
\end{eqnarray}
\item\label{iv}
For $\nu_V^W$-almost all $\beta$, all $\gamma>0$ and all $i_0\in V$, we have,
\begin{itemize}
\item
the Markov process \(P_{i_0}^{\beta,\gamma,i_{0}}\)
is transient if and only if \(\psi_\beta(j)>0\) for all \(j\in V\),
\item
the Markov process \(P_{i_0}^{\beta,\gamma,i_{0}}\) is recurrent if and only if \(\psi_\beta(j)=0\) for all \(j\in V\).
\end{itemize}

\end{enumerate}
\end{theorem}
\noindent
{\bf N.B.}: Note that $P^{\beta,\gamma,i_0}_x$ is well defined for $\nu_V^W$-almost all $\beta$ and all $\gamma>0$ by \ref{convergence_i} and \ref{convergence_ii}. 
\begin{notation}
We denote by
\begin{align}
\label{nu_beta_gamma}
\nu_V^{W}(d\beta,d\gamma):=\nu_V^{W}(d\beta)\otimes\frac{\mathds{1}_{\gamma>0}}{\sqrt{\pi\gamma}}e^{-\gamma}d\gamma 
\end{align}
 the probability distribution which appears in \eqref{rep_VRJP}, under which $\gamma$ is $Gamma(\demi, 1)$-distributed and independent of $\beta$. 
 In general, we simply write $G(i,j)$ for $G_{\beta,\gamma}(i,j)$ and 
 consider it as a random variable on the probability space $(\ddd_V^W\times \R^*_+, \bbb(\ddd_V^W)\otimes\bbb(\R_+^*), \nu_V^W(d\beta,d\gamma))$.
\end{notation}
\begin{remark}
When the VRJP is recurrent, \(G=\hat G\), and the representation of the VRJP \eqref{rep_VRJP} only involves the variable $\beta$ and not $\gamma$.
\end{remark}
\begin{remark}
The representation (\ref{jumping-G}) extends to infinite graphs the representation provided in \cite[Theorem 2]{STZ} for finite graphs.
An interesting new feature appears in the transient regime, where the generalized eigenfunction \(\psi\) and the extra gamma random variable
enter the expression of \(G(i,j)\). As it appears in the proof, the eigenfunction \(\psi\) can be interpreted as the mixing field of a VRJP starting from infinity.
\end{remark}

Denote by \(\tau_{i_0}^{+}=\inf\{t\ge 0, \;\; Z_t=i_0, \; \exists  s<t \hbox{ s.t. } Z_s\neq i_0\}\) the first return time to \(i_0\) by \((Z_t)_{t\ge 0}\).
The point \ref{iv} of the previous theorem is in fact a consequence of the following more precise assertion.
\begin{proposition}
\label{coro-escape-proba}
We have, for $\nu_V^W$-almost all $\beta$, for all $\gamma>0$ and $i_0$, $i\in V$,
   \[
    P_{i}^{\beta,\gamma,i_{0}}(\tau_{i_0}^{+}=\infty)=
\begin{cases}
\frac{\psi(i_{0})^{2}}{4\gamma \tilde{\beta}_{i_{0}}\hat{G}(i_{0},i_{0})G(i_{0},i_{0})}, & \hbox{if $i=i_{0}$,}\\
\frac{\psi(i_{0})}{2\gamma}\frac{\hat{G}(i_{0},i_{0})\psi(i)-\hat{G}(i_{0},i)\psi(i_{0})}{\hat{G}(i_{0},i_{0})G(i_{0},i)}, & \hbox{ if $i\neq i_{0}$,}
\end{cases}
  \]
where \(\tilde{\beta}_{i_{0}}=\sum_{j\sim i_{0}}\frac{1}{2}W_{i_{0},j}\frac{G(i_{0},j)}{G(i_{0},i_{0})}\). In particular, \(\psi(i_0)=0\) if and only if
\(P_{i_0}^{\beta,\gamma,i_{0}}(\tau_{i_0}^{+}=\infty)=0\).
\end{proposition}
Using Doob's \(h\) transform, the law of the process \((Z_t)\) conditioned on the event \(\{\tau_{0}^{+}<\infty\}\)  or  \(\{\tau_{0}^{+}=\infty\}\) can be computed and
takes a rather nice form,
both under the law $\P^{VRJP}_{i_0}$, or under the law $P_{i_0}^{\beta,\gamma,i_{0}}$ for $\nu_V^W$-almost all $\beta$. We provide these formulae in Section~\ref{ss_h-transform}.

A natural question that emerges from point \ref{iv} of the theorem is that of a 0-1 law for transience/recurrence. We provide an
answer below in the case of vertex transitive graphs with conductances. We say that \((\ggg, W)\)  is vertex transitive
if the group of automorphisms of \(\ggg\) that leaves invariant \((W_{i,j})\) is transitive on vertices. In particular, it is the case for the cubic lattice
\(\Z^d\) with constant conductances \(W_{i,j}=W\). Denote by \(\aaa\) the group of automorphisms that leave \(W\) invariant.
\begin{proposition}\label{0-1law}
If \((\ggg,W)\) is vertex transitive and \(\ggg\) is infinite, then under the distribution \(\nu_V^W(d\beta)\), the random variables \((\beta_i)_{i\in V}\), \((\psi(i))_{i\in V}\), \((\hat G(i,j))_{i,j\in V}\) are stationary and ergodic for the
group of transformations \(\aaa\). Moreover,  the VRJP is either recurrent or transient, i.e.
\[
\P^{\operatorname{VRJP}}_{i_0}( \hbox{ every vertex is visited i.o.\ })=1 \hbox{ or } \P^{\operatorname{VRJP}}_{i_0}( \hbox{ every vertex is visited f.o. })=1.
\]
In the first case \(\psi(i)=0\) for all \(i\in V\), a.s., in the second case \(\psi(i)>0\) for all \(i\in V\), a.s.
\end{proposition}
\noindent N.B: The action of \(\aaa\) on \(\hat G\) is \((\tau \hat G)(i,j)= \hat G(\tau i, \tau j)\) for \(\tau \in \aaa\).

\subsection{Relation with random Schr\"odinger operators}
Let us now relate Theorem~\ref{convergence} to the properties of the random Schr\"odinger operator $H:\beta\mapsto H_\beta$ associated
with the random potential \((\beta_j)\) under the law $\nu_V^W$,
defined in \eqref{def_Hbeta} and Proposition~\ref{kolmogorov}.
\begin{theorem}
\label{main}
Under $\nu_V^W(d\beta)$:
\begin{enumerate}[label=(\roman*)]
\item The spectrum of \(H\) is a.s. included in \([0,\infty)\).
\item The operator \(\hat G\) is the inverse of \(H\) in the following sense: for all \(i,j\in V\), a.s.
\[
\hat G(i,j)=\lim_{\epsilon>0, \epsilon \to 0} (H+\epsilon)^{-1}(i,j).
\]

\item We have \((H \psi)(i)=0\) a.s.\ for all \(i\in V\).

\item In the case of the grid \(\Z^d\) and when \(W_{i,j}=W\) is constant, \((\hat G(i,j))\) and
\((\psi(i))\) are stationary ergodic for the spacial shift. Moreover,  in the transient case, \(\psi\) is a.s. a
positive generalized eigenfunction  with eigenvalue 0 in the sense that \(H \psi=0\) and \(\psi\) has at most
polynomial growth.
More precisely, for all \(p>d\) and $C>0$,  a.s. there exists a random integer  \(K>0\) such that 
\[
\vert \psi(i) \vert\le C \| i\|_\infty^p \;\;\; \forall i\in \Z^d \hbox{ such that }\; \|i\|_\infty\ge K.
\]
\end{enumerate}
\end{theorem}

\subsection{Functional central limit theorem}
We denote by \((\tilde{Z}_{n})_{n\in \N}\) the discrete time process that describes the successive jumps of \((Z_{t})_{t\in \R_+}\). From Theorem~\ref{convergence}~\ref{convergence_iii}, under $\P^{VRJP}_{i_0}$,  \(\tilde{Z}_{n}\) is a mixture of Markov chains starting from $i_0$ and with conductances
\begin{equation}
\label{mixture-discrete-VRJP}
W_{i,j}G(i_{0},i)G(i_{0},j),
\end{equation}
under the probability distribution $\nu_V^W(d\beta,d\gamma)$.

We prove below a functional central limit theorem for the discrete time VRJP on $\Z^d$, $d\ge 3$, at weak reinforcement (i.e.\ for \(W\) large enough).

\begin{theorem}
\label{thm-clt}
Consider the 
cubic graph \(\mathbb{Z}^{d}\), \(d\ge 3\), with constant conductances \(W_{i,j}=W\).
Denote
\[\tilde{Z}^{(n)}_{t}=\frac{\tilde{Z}_{[nt]}}{\sqrt{n}}.\]
There exists \(\lambda_2>0\) such that if \(W>\lambda_2\),
the discrete time VRJP
satisfies a functional central limit theorem, i.e.
under \(\P_{0}^{\text{VRJP}}\), for any real $0<T<\infty$,
\((\tilde{Z}^{(n)}_{t})_{t\in [0,T]}\) converges in law (for the Skorokhod topology)
to a \(d\)-dimensional Brownian motion \((B_{t})_{t\in [0,T]}\) 
with non degenerate isotropic diffusion matrix \(\sigma^2 Id\), for some \(0<\sigma^2<\infty\).
\end{theorem}

\subsection{Consequences for the Edge Reinforced Random Walk (ERRW)}\label{results-ERRW}

The Edge Reinforced Random Walk (ERRW) is a famous discrete time process introduced in 1986 by Coppersmith and Diaconis,~\cite{coppersmith1987random,rolles2000edge}.

Endow the edges of the graph $\ggg=(V,E)$ with some positive weights \((a_e)_{e\in E} \).
Let \((X_n)_{n\in\mathbb{N}}\) be a random process that takes values in \(V\), and let
\(\mathcal{F}_n=\sigma(X_0,\ldots,X_n)\) be the filtration of its past. For any \(e\in E\), \(n\in\mathbb{N}\), let
\begin{equation}\label{Zn}
N_n(e)=a_e+ \sum_{k=1}^n\mathds{1}_{\{\{X_{k-1},X_k\}=e\}}
\end{equation}
be the number of crossings of the (undirected) edge \(e\) up to time \(n\) plus the initial weight \(a_e\).

Then \((X_n)_{n\in\mathbb{N}}\) is called Edge Reinforced Random Walk (ERRW) with starting point \(i_0\in V\) and weights \((a_e)_{e\in E}\),
if  \(X_0=i_0\) and,  for all \(n\in\mathbb{N}\),
\begin{equation}
\label{def-errw}
\mathbb{P}(X_{n+1}=j~|~\mathcal{F}_n)=\mathds{1}_{\{j\sim X_n\}}\frac{N_n(\{X_n,j\})}
{\sum_{k\sim X_n} N_n(\{X_n,k\})}.
\end{equation}
We denote by \(\mathbb{P}^{ERRW}_{i_0}\) the law of the ERRW starting from the initial vertex \(i_0\). We will assume that the ERRW is defined on the canonical space $V^\N$, i.e. that $(X_n)_{n\in \N}$ is the canonical process on $V^\N$.

Important progress has been done in the last ten years in the understanding of this process, c.f. e.g. \cite{angel2014localization,disertori2014transience,merkl2009recurrence,sabot2011edge}.
In particular, in was proved in 2012 by Sabot, Tarr\`es, \cite{sabot2011edge}, and Angel, Crawford, Kozma,
\cite{angel2014localization}, on any graph with bounded degree at strong reinforcement (i.e.\ for \(a_e<\tilde\lambda_0\) for
some fixed \(\tilde\lambda_0>0\)) that the ERRW is a mixture of positive recurrent
Markov chains. It was proved by Disertori, Sabot, Tarr\`es \cite{disertori2014transience} that on \(\Z^d\), \(d\ge 3\), the ERRW is transient at weak reinforcement
(i.e.\ for \(a_e>\tilde\lambda_1\) for some fixed \(\tilde\lambda_1<\infty\)).

From \cite[Theorem 1]{sabot2011edge}, we know that the ERRW has the law of a VRJP in independent conductances.
More precisely, consider \((W_e)_{e\in E}\) as independent random variables with gamma distribution with parameters \((a_e, 1)\).
Consider the VRJP in conductances \((W_e)_{e\in E}\) and its underlying discrete time process \((\tilde Z_n)\). Then the annealed
law of \((\tilde Z_n)\) (after expectation with respect to \(W\)) is that of the ERRW \((X_n)\) with initial weights \((a_e)\).
Hence, we can apply Theorem~\ref{convergence} at fixed \(W\) and then integrate on \(W\).
We thus consider the joint law  \(\tilde \nu_V^{a}(dW,d\beta,d\gamma)\) on $(\R_+^*)^E\times (\R_+^*)^V\times\R_+^*$ obtained from $\nu_V^W(d\beta,d\gamma)$ after randomization with respect to $W$. More formally, let $\tilde\nu^{a}_V(dW)$ be the probability distribution on $(\R_+^*)^E$ such that under $\tilde\nu^a_V(dW)$ the random variables $W\mapsto W_e$ are independent with gamma distribution with parameters $(a_e,1)$, then $\tilde\nu_V^a(dW,d\beta,d\gamma)$ is the probability distribution on $(\R_+^*)^E\times (\R_+^*)^V\times\R_+^*$ such that for any bounded measurable test function \(F\), 
\[
\int F(W,\beta,\gamma)  \tilde \nu_V^{a}(dW,d\beta,d\gamma)=\int\left( \int F(W,\beta,\gamma) \nu_V^{W}(d\beta,d\gamma)\right)\tilde\nu_V^a(dW).
\]
In the sequel, \( \tilde \nu_V^{a}(dW,d\beta)\), \( \tilde \nu_V^{a}(d\beta)\) will denote the corresponding marginal distributions, and $\tilde\nu^a_V(dW)$ is the $W$ marginal. (By definition, $\tilde \nu_V^{a}(dW,d\beta)$ is supported on the set of $(W,\beta)$ such that $\beta\in \ddd_V^W$.)
 From Theorem~\ref{convergence}, we see that the ERRW starting from \(i_0\) is a mixture
of reversible Markov chains with conductances
\begin{eqnarray}\label{rep-ERRW}
x_{i,j}=W_{i,j}G(i_0,i)G(i_0,j),
\end{eqnarray}
where \(G\) is defined in Theorem~\ref{convergence}, and \((W,\beta,\gamma)\) are distributed according to
\(\tilde \nu_V^{a}(dW,d\beta,d\gamma)\). More formally, if $\tilde P^x_{i_0}$ denotes the law of the Markov chain starting at $i_0$ and with conductances $(x_{i,j})_{i\sim j}$, then
$$
\P^{ERRW}_{i_0}(\cdot)=\int \tilde P^{x}_{i_0}(\cdot)\tilde\nu_V^a(dW,d\beta,d\gamma).
$$

An important point is that we keep the 1-dependence of the field \(\beta\),  after taking expectation with respect to
\(W\).
\begin{proposition}\label{indep-ERRW}
Under \(\tilde \nu_V^{a}(d\beta)\), \((\beta_j)_{j\in V}\) is 1-dependent: if \(U, U' \subset V\) are such that \(\operatorname{d}_{\mathcal{G}}(U, U’)\geq 2\), then \((\beta_i)_{i\in U}\) and \((\beta_j)_{j\in U’}\) are independent.
\end{proposition}
\begin{proof} Indeed, from Proposition~\ref{kolmogorov}, the Laplace transform of \((\beta_i)_{i\in U}\) under $\nu_V^W(d\beta)$ only involves the conductances \(W_{i,j}\) for
\(i\) or \(j\) in \(U\). This implies that, if \(\operatorname{d}_{\mathcal{G}}(U, U’)\geq 2\), the joint Laplace transform of \((\beta_i)_{i\in U}\) and \((\beta_i)_{i\in U'}\) is still the product of Laplace transforms
even after taking expectation with respect to the random variables \((W_e)\), i.e. under $\tilde \nu_V^a(\d\beta)$.
\end{proof}
This yields a counterpart of Proposition~\ref{0-1law} for the ERRW.
\begin{proposition}\label{0-1law-ERRW}
Assume \((\ggg,(a_{i,j}))\) is vertex transitive with automorphism group \(\aaa\), and \(\ggg\) infinite. Then under the distribution \(\tilde\nu_V^{a}(dW,d\beta)\),
the random variables \((W_e)_{e\in E}\), \((\beta_i)_{i\in V}\), \((\psi(i))_{i\in V}\), \((\hat G(i,j))_{i,j\in V}\) are stationary and ergodic for the
group of transformations \(\aaa\). Moreover,  the ERRW is either recurrent or transient, i.e.
\[
\P^{ERRW}_{i_0}( \hbox{ every vertex is visited i.o.\ })=1, \hbox{ or } \P^{ERRW}_{i_0}( \hbox{ every vertex is visited f.o. })=1.
\]
In the first case \(\psi(i)=0\) for all \(i\in V\), a.s., in the second case \(\psi(i)>0\) for all \(i\in V\), a.s.
\end{proposition}
\noindent N.B: The action of \(\aaa\) on \(\hat G\) and \(W\)  is \((\tau \hat G)(i,j)= \hat G(\tau i, \tau j)\), \(\tau W_{i,j}=W_{\tau i, \tau j}\)
 for \(\tau \in \aaa\).
 \begin{remark}
In \cite{merkl2007random}, it was proved on infinite graphs that the ERRW is a mixture of Markov chains,
obtained as a weak limit of the mixing law of the ERRW on finite approximating graphs.
The difference in the representation we give in (\ref{rep-ERRW}) is that the random variables \(\psi\), \(\hat G\) are obtained as almost sure limits
and hence are measurable functions of the random variables \(\beta\).
This yields stationarity and ergodicity, which are the key ingredients in the 0-1 law, and in forthcoming Theorems~\ref{FCLT-ERRW}
and~\ref{rec-ERRW}.
\end{remark}
\begin{remark}
It seems that this 0-1 law is new, both for the VRJP and the ERRW.
In \cite{merkl2007random}, it was proved that if the ERRW comes back with probability 1 to its starting point then
it visits infinitely often all points, a.s., which is a weaker result.
This was proved using the representation of the ERRW as mixture of Markov chains
of \cite{merkl2007random}. (A short proof of this last result can also be given, c.f. \cite{tournier2009note}.)
 \end{remark}

 We now give a counterpart of Theorem~\ref{thm-clt} for the ERRW. It is a consequence of Theorem~\ref{convergence} and of
 the delocalization result proved by Disertori, Sabot, Tarr\`es in~\cite{disertori2014transience}.
 \begin{theorem}\label{FCLT-ERRW}
Consider the  
cubic graph \(\mathbb{Z}^{d}\), \(d\ge 3\), with constant weights \(a_{i,j}=a\). Denote
\[{X}^{(n)}_{t}=\frac{{X}_{[nt]}}{\sqrt{n}}.\]
There exists \(\tilde\lambda_2>0\) such that if \(a>\tilde\lambda_2\),
the ERRW
satisfies a functional central limit theorem, i.e.
under \(\P_{0}^{\text{ERRW}}\), for any real $0<T<\infty$,
\(({X}^{(n)}_{t})_{t\in [0,T]}\) converges in law (for the Skorokhod topology)
to a \(d\)-dimensional Brownian motion \((B_{t})_{t\in [0,T]}\)
 with non degenerate isotropic diffusion matrix \(\sigma^2 Id\), for some \(0<\sigma^2<\infty\).
\end{theorem}
Finally, we can deduce recurrence of the ERRW in dimension 2 from Theorem \ref{convergence}, Proposition \ref{0-1law-ERRW} and 
the estimates obtained by Merkl and Rolles in \cite{merklbounding,merkl2009recurrence}\footnote{We
are grateful to Franz Merkl and Silke Rolles for a useful discussion on that subject}.
\begin{theorem}\label{rec-ERRW}
The ERRW \(({X}_{n})_{n\ge 0}\) on \(\mathbb{Z}^{2}\) with constant weights \(a_{i,j}=a\) is a.s.\ recurrent, i.e.
\[
\P^{ERRW}_{0}\left( \hbox{ every vertex is visited infinitely often }\right)=1.
\]
\end{theorem}
In \cite{merklbounding,merkl2009recurrence}, by a Mermin-Wagner type argument, Merkl and Rolles proved a polynomial decrease of the form
\begin{eqnarray}\label{polynomial}
\E\left( \left({x_\ell\over x_0}\right)^{{1\over 4}} \right)\le c(a) \vert \ell\vert^{-\xi(a)},
\end{eqnarray}
for some constants \(c(a)>0\), \(\xi(a)>0\), depending only on \(a\), and
where \(x_{\ell}\) is the conductance at the site \(\ell\) for the mixing measure of the ERRW, uniformly for a sequence
of finite approximating graphs.
When \(0<\xi<1\), it does not give by itself enough information to prove recurrence.
It was used in the case of a diluted 2-dimensional graph to prove positive recurrent at strong reinforcement.
The extra information given by the representation (\ref{rep-ERRW}) and the stationarity of \(\psi\), implies that
the polynomial estimate (\ref{polynomial}) is incompatible with \(\psi(i)>0\) and hence is incompatible with transience.
Detailed arguments are provided in Section~\ref{ss_proof-rec}.
\begin{remark}
We expect similarly that the 2-dimensional VRJP with constant conductances $W_{i,j}=W>0$ is recurrent.
This would be implied by an estimate of the type  (\ref{polynomial}) for the mixing field of the VRJP, which is still not
available.
More precisely, we can see from the proof of Theorem \ref{rec-ERRW} in Section \ref{ss_proof-rec}, that recurrence of the
2-dimensional VRJP would be implied by Theorem \ref{convergence}, Proposition \ref{0-1law}, and an estimate of the type
\[
\E\left(e^{\eta(u_\ell-u_0)}\right)\le \epsilon(\| \ell \|_{\infty}),
\]
for $\eta>0$ and $\epsilon(n)$ a positive function such that $\lim_{n\to \infty } \epsilon(n)=0$, where
$(u_j)$ is the mixing field of the VRJP starting from 0 (c.f. Theorem\ref{thm_vrjp_u_ST}) on finite boxes with wired boundary condition
as in Section \ref{ss_Kolmogorov}. 
We learned from G. Kozma and R. Peled that they have a proof of
such an estimate.
\end{remark}

\subsection{Open questions}
The most important question certainly concerns the relation between the properties of the VRJP and the spectral properties of the
random Schr\"odinger operator \(H_\beta\).
For example on \(\Z^d\) with constant weights \(W_{i,j}=W\), is recurrence/transience of the VRJP related to the
localized/delocalized regimes of \(H_\beta\)? A more precise question would be: does the transient regime
of the VRJP coincide with the existence of extended states at least at the bottom of the spectrum of \(H_\beta\)?
It might at first seem inconsistent to expect extended states at the bottom of the spectrum since the Anderson model with
i.i.d.\ potential is expected to be localized at the edges of the spectrum
(a fact which is proved in several cases). But this localization is a consequence of Lifshitz tails, and there are good reasons to
expect that Lifshitz tails fail for the potential \(\beta\), which is not i.i.d.\ but 1-dependent. Indeed, the bottom of the spectrum of \(H_\beta\) is 0, it does not coincide with
the minimum of the support of the distribution of \(2\beta\) translated by the spectrum of \(-P\), as it is the case for i.i.d.\ potential.
In fact, on a finite set, the minimum of the spectrum is reached on the set \(\det(2\beta-P)=0\) which is a set of codimension 1, hence it is "big".

Another natural question concerns the uniform integrability of the martingale \(\psi^{(n)}(i)\).
Let us ask a more precise question: is it true (at least for \(\Z^d\) with constant weights) that
transience of the VRJP implies that the martingale \(\psi^{(n)}(i)\) is bounded in \(L^2\)?
It is quite natural to expect such a property from relation (\ref{eq-mart-crochet}) since
\(\hat G^{(n)}(i,i)\) appears to be the quadratic variation of \(\psi^{(n)}(i)\).
This would have several consequences. Firstly, it would imply that in dimension \(d\ge 3\),
the VRJP satisfies a functional central limit theorem as soon as the VRJP is transient, by the same argument as that of
the proof of Theorem~\ref{thm-clt}.
It would also imply directly that the VRJP is recurrent as soon as the reversible Markov chain in conductances \((W_{i,j})\)
is recurrent, if the group of automorphisms of \((\ggg,W)\) is transitive.
Indeed, assume that the property is true and the VRJP is transient.
By Theorem~\ref{convergence}, the discrete time process \((\tilde Z_n)\) would be represented as a mixture of
reversible Markov chains with conductances
\(W_{i,j} G(0,i)G(0,j)\).
From Proposition~\ref{coro-escape-proba} applied to $i_0=0$, we have that
\[
{\hat G(0,i)\over \hat G(0,0)} \le {\psi(i)\over \psi(0)}.
\]
Hence, \((\tilde Z_n)\) is equivalently a mixture of Markov chains with conductances
\[
\frac{\psi(0)^2}{G(0,0)^2} W_{i,j} {G(0,i)G(0,j)}\le W_{i,j} {\psi(i)\psi(j)}.
\]
But \((\psi(i))\) is stationary ergodic,  if \(\psi_0\) is squared integrable, we would have
\[\E_{\nu_V^W}(W_{i,j} {\psi(i)\psi(j)})\le CW_{i,j}
\]
for some constant \(C>0\).
Usual arguments imply that the Markov chain in conductance
\(W_{i,j}\psi(i)\psi(j)\) is recurrent if the Markov chain in conductances \((W_{i,j})\) is recurrent (c.f. e.g. Exercise 2.75, \cite{lyons-peres}). We arrive at a contradiction.

\subsection{Organization of the paper}
In Section~\ref{finite-graphs}, we gather several results in the case of finite graphs, in particular we recall the main results of \cite{STZ}. In Section~\ref{ss_wired},
we define the important notion of restriction with wired boundary condition and the compatibility property.
Section~\ref{ss_martingale} is the key step in the paper where the martingale property is proved.
In Section~\ref{ss_thm1}, we prove Theorem~\ref{convergence}, Propositions~\ref{coro-escape-proba} and~\ref{0-1law} and Theorem~\ref{main}.
In Section~\ref{ss_h-transform}, we provide extra computations of \(h\)-transforms.
In section~\ref{ss_proof-rec}, we prove recurrence of ERRW in dimension 2 for all initial constant  weights.
In Section~\ref{ss_clt}, we prove functional central limit theorems for the VRJP and the ERRW, Theorems~\ref{thm-clt} and~\ref{FCLT-ERRW}.

\section{The random potential \(\beta\) on finite graphs}\label{finite-graphs}
In this section we assume that \(\mathcal{G}=(V,E)\) is a finite graph and gather several results in this case. Recall that 
every undirected edge \(e=\{i,j\}\) is labeled with a positive conductance \(W_e=W_{i,j}\). In the case of a finite graph, the Schr\"odinger operator $H_\beta$ defined in \eqref{def_Hbeta} can be represented by the $V\times V$-matrix given by
\[H_{\beta}(i,j)=
\begin{cases}
  2\beta_i, & i=j,\\
  -W_{i,j}, & i\neq j, \; i\sim j,\\
  0,&{otherwise,}
\end{cases}
\]
and the set $\ddd_V^W$ defined in \eqref{DVW} is the set of potentials $\beta$ such that $H_\beta$ is positive definite.

\subsection{The probability distribution $\nu_V^W$ on finite graphs and relation to the VRJP}
We recall \cite[Theorem 1]{STZ}, which defines the probability distribution $\nu_V^W(d\beta)$ by its density on any finite graph.
\begin{theoA}[Theorem 1, Definition 1 and Proposition 1 of \cite{STZ}]
\label{thm_potential_fini}
 Let \((\mathcal{G},(W_e)_{e\in E})\) be a  finite graph with conductances. The measure below is a probability on 
 $\ddd_V^W$:
 \begin{equation}
  \label{eq:densitybeta}
  \nu_V^{W}(d\beta):=\mathds{1}_{H_{\beta}>0}\left(\frac{2}{\pi}\right)^{|V|/2}\exp({-\sum_{i\in V}\beta_i+\sum_{e\in E}W_e})\frac{d\beta_{V}}{\sqrt{\det H_{\beta}}}
\end{equation}
with  \(d\beta_{V}=\prod_{i\in V}d\beta_{i}\), and
where $H_\beta>0$ means that $H_\beta$ is positive definite.

The Laplace transform of the probability distribution $\nu_V^W(d\beta)$ is given, for all \((\lambda_i)\in \R_+^V\), by
\begin{equation}
  \label{laplac_trans}
\int e^{-\left<\lambda,\beta\right>}\nu_V^{W}(d\beta)=\exp\left(-\sum_{i\sim j}W_{i,j}(\sqrt{(\lambda_{i}+1)(\lambda_{j}+1)}-1)\right) \prod_{i\in V}\frac{1}{\sqrt{\lambda_{i}+1} }.
\end{equation}
Moreover, we have the following properties: under $\nu_V^W(d\beta)$,
\begin{itemize}
\item
{\it(1-dependence)}: if \(U,U’\subset V\) are such that \(\operatorname{d}_{\mathcal{G}}(U, U’)\geq 2\), then the random variables \(\beta\mapsto \beta_U\) and \(\beta\mapsto\beta_{U’}\) are independent,
\item
{\it(Reciprocal inverse Gaussian marginals)}
for $i\in V$, the random variable \(\beta\mapsto \frac{1}{2\beta_i}\) has an inverse Gaussian distribution with parameter \((\frac{1}{W_{i}},1)\) where \(W_i=\sum_{j\sim i} W_{i,j}\).
\end{itemize}
\end{theoA}
If we apply formula \eqref{laplac_trans} to \((\lambda_i)\in \R_+^V\) such that $\lambda_{V\setminus U}=0$ for a subset $U\subset V$, we find the expression of Proposition~\ref{kolmogorov}. Hence, it implies Proposition~\ref{kolmogorov} in the case of a finite graph.

The field \(\beta\) is closely related to the VRJP, as shown in the next two theorems. 
In~\cite{sabot2011edge}, it is shown that the VRJP in exchangeable time scale defined in Section~\ref{sec_Notations} is a mixture of Markov jump processes,
more precisely:
\begin{theoA}[Theorem 2 of \cite{sabot2011edge}]
\label{thm_vrjp_u_ST}
Assume \(V\) finite. The following measure is a probability distribution on the set \(\{(u_i)_{i\in V}\in \mathbb{R}^V,\; u_{i_0}=0\}\):
\begin{equation}
  \label{density_u}
\mathcal{Q}_{i_{0}}^{W}(du)=\frac{1}{\sqrt{2\pi}^{|V|-1}} \exp\left(-\sum_{i\in V}u_i-\sum_{i\sim j}W_{i,j}(\cosh(u_i-u_j)-1)\right)\sqrt{D(W,u)}du_{V\setminus\{i_{0}\}}
\end{equation}
where \(du_{V\setminus\{i_{0}\}}=\prod_{i\in V\setminus\{i_0\}} du_i\) and
\(
D(W,u)=\sum_{T\in \ttt} \prod_{\{i,j\}\in T} W_{i,j}e^{u_i+u_j}
\), where the sum is over \(\ttt\), the set of spanning trees of the graph \(\ggg\).

For $(u_i)_{i\in V}\in \R^V$, we denote by $P^{(u)}_{i_0}$ the law of the Markov jump process starting at vertex $i_0$ and with jump rates from $i$ to $j$ given by
  \[
  \frac{1}{2}W_{i,j}e^{u_j-u_i}.
  \] 
The law of the VRJP in exchangeable time scale 
starting at \(i_0\) is a mixture of Markov jump processes, with mixing law given by
$$
\P^{VRJP}_{i_0}(\cdot)=\int P^{(u)}_{i_0}(\cdot) \mathcal{Q}_{i_{0}}^{W}(du).
$$
\end{theoA}
\begin{remark}
  By the matrix-tree theorem, \(D(W,u)\) is any diagonal minor of the \(|V|\times |V|\) matrix \((m_{i,j})\) with coefficients
\[m_{i,j}=
\begin{cases}
0,& \text{ if } i\not\sim j,\; i\neq j,\\
 - W_{i,j}e^{u_{i}+u_{j}}, & \text{ if } i\sim j,\; i\neq j,\\
  \sum_{k\in V, k\sim i}W_{i,k}e^{u_{i}+u_{k}}, & \text{ if }i=j.
\end{cases}
\]
\end{remark}
\begin{remark}
The probability measure \(\mathcal{Q}_{i_{0}}^{W}(du)\) appeared previously to \cite{sabot2011edge} in a rather different context in the work of Disertori, Spencer, Zirnbauer, \cite{disertori2010quasi}.
In particular, the fact that \(\mathcal{Q}_{i_{0}}^{W}(du)\) is a probability measure was proved there as a consequence of a Berezin identity applied to
a supersymmetric extension of that measure.
\end{remark}
On finite graphs, the random environment \((u_i)\) of the previous theorem can be represented thanks to the Green function of the random potential \((\beta_{i},i\in V)\) distributed according to $\nu_V^W(d\beta)$.
Let us first recall \cite[Proposition 1]{STZ}.
\begin{propA}[Proposition 1 of \cite{STZ}]
\label{prop_u_beta_STZ}
Assume \(V\) finite.
For $\beta \in \ddd_V^W$, we denote by 
\[G_\beta:=(H_\beta)^{-1}\]
the Green function of the Schr\"odinger operator \(H_\beta\). 
For $\beta\in \ddd_V^W$, $i,j\in V$, we define $u_\beta(i,j)$ by
\begin{equation}
\label{u(i,j)}
e^{u_\beta(i,j)}=\frac{G_\beta(i,j)}{G_\beta(i,i)}.
\end{equation}
For $i_0\in V$, $(u_\beta(i_0,j))_{j\in V}$ is the unique solution of the equation
\begin{align}
\label{eq_psi}
\begin{cases}
u_\beta(i_0,i_0)=0,&
\\
\sum_{j\sim i} W_{i,j}e^{u_\beta(i_0,j)-u_\beta(i_0,i)}= \beta_i, &\hbox{if $i\neq i_0$}.
\end{cases}
\end{align}
In particular,
the function \(\beta\mapsto (u_\beta (i_0,j))_{j\in V}\) is \((\beta_j)_{j\in V\setminus\{i_0\}}\)-measurable.
Moreover, 
for all $\beta\in \ddd_V^W$,
\begin{equation}
\label{beta-and-G}
\beta_{i_0}=\frac{1}{2G_\beta(i_{0},i_{0})}+\frac{1}{2}\sum_{j\sim i_0}W_{i,j}e^{u_\beta(i_0,j)-u_\beta(i_0,i_0)}.
\end{equation}
\end{propA}
As usual, we simply denote by $G(i,j)$ and $u(i,j)$
the associated random variables on the probability space $(\ddd_V^W,\bbb(\ddd_V^W),\nu_V^W)$.
Let us now recall \cite[Theorem 3]{STZ}.
\begin{theoA}[Theorem 3 of \cite{STZ}]
\label{thm_u_beta_STZ}
Assume \(V\) finite.
For all \(i_0\in V\), under the probability $\nu_V^W(d\beta)$,
\begin{enumerate}[label=(\roman*)]
\item
the random field \((u(i_0,j))_{j\in V}\) has the distribution \(\mathcal{Q}_{i_{0}}^{W}\)
of Theorem~\ref{thm_vrjp_u_ST},
\item
\({1\over 2 G(i_0,i_0)}\) has a gamma distribution with parameters \((1/2, 1)\),
\item
\(G(i_0,i_0)\) is independent of \((\beta_j)_{j\neq i_0}\), hence independent of the field \((u(i_0,j))_{j\in V}\).
\end{enumerate}
\end{theoA}
\begin{remark}
  Here we only consider the VRJP with initial local time \(1\), in fact, the above correspondence between \(\beta\) and VRJP still holds for the process starting with any positive local times \((\phi_{i},i\in V)\), in such case, there is a corresponding density \(\nu_V^{W,\phi^2}\), which is defined in~\cite{STZ}, see Definition~1 and Theorem~3. We choose here to normalize the initial local
  time to 1 since it is equivalent to the general case by a change of time and $W$, see \cite{STZ} Appendix~B.
\end{remark}
Combining Theorem~\ref{thm_vrjp_u_ST} and Theorem~\ref{thm_u_beta_STZ}, it gives a representation of the VRJP in exchangeable time scale starting from different points in terms of the probability on random potentials $\nu_V^W$. We state this representation below.
\begin{corollary}\label{rep_VRJP_fini}
Assume $V$ finite. For $\beta\in \ddd_V^W$, define $P^{\beta,i_0}_x$ as the law of the Markov jump process starting from $x$ and with jump rates from $i$ to $j$ given by
$$
\demi W_{i,j}{G_\beta(i_0,j)\over G_\beta(i_0,i)}.
$$
Then, the VRJP in exchangeable time scale is a mixture of these Markov jump processes,
 \begin{eqnarray}\label{rep_VRJP_finite}
\P^{\text{VRJP}}_{i_{0}}(\ \cdot\ )=\int P_{i_0}^{\beta,i_{0}}(\ \cdot \ ) \nu_V^{W}(d\beta).
\end{eqnarray}
\end{corollary}

\subsection{Representation as a sum on paths.}
\label{ss_sum_paths}

We call {\it path} in \(\mathcal{G}\) from \(i\) to \(j\) a finite sequence \(\sigma=(\sigma_0, \ldots, \sigma_m)\)
in \(V\) such that \(\sigma_0=i,\ \sigma_m=j\) and \(\sigma_k\sim \sigma_{k+1}\), for \(k=0, \ldots, m-1\).
The length of \(\sigma\) is defined by \(|\sigma|=m\). We denote by \(\mathcal{P}_{i,j}^{V}\) be the collection of paths in \(V\) from \(i\) to \(j\), and 
\(\bar{\mathcal{P}}_{i,j}^{V}\) be the collection of paths
\(\sigma=(\sigma_0=i, \ldots, \sigma_m=j)\)  in \(V\) from \(i\) to \(j\) such that \(\sigma_k\neq j,k=0,\ldots,m-1\).
For a path $\sigma$ and for $\beta\in \ddd_V^W$, we set
\begin{equation}
\label{eq-defi-W_sigma-beta-sigma}
W_\sigma=\prod_{k=0}^{m-1} W_{\sigma_{k},\sigma_{k+1}}, \;\;\;
(2\beta)_\sigma=\prod_{k=0}^{m} (2\beta_{\sigma_{k}}), \;\;\; (2\beta)^-_\sigma=\prod_{k=0}^{m-1} (2\beta_{\sigma_{k}}).
\end{equation}
For the trivial path \(\sigma=(\sigma_{0})\), we define \(W_{\sigma}=1\), \((2\beta)_{\sigma}=2\beta_{\sigma_{0}}\),
\((2\beta_{\sigma})^{-}=1\). (Note that these definitions make sense also in the case of infinite graphs.)

 The following representation of the Green function \(G(\cdot,\cdot)\) as a sum on paths will be convenient.
\begin{proposition}
\label{fini_green} Assume that $V$ is finite.
For all 
$\beta\in \ddd_V^W$, we have, with the notations of Theorem~\ref{prop_u_beta_STZ},
\begin{equation}
\label{eq-tilG-pathsum}
G_\beta(i,j)=\sum_{\sigma\in \mathcal{P}_{i,j}^V}\frac{W_{\sigma}}{(2\beta)_{\sigma}},
\;\;\;\; 
\;\;\;\;
\exp(u_\beta(i,j))=\sum_{\sigma\in \bar{\mathcal{P}}_{j,i}^{V}}\frac{W_{\sigma}}{(2\beta)_{\sigma}^-}.
\end{equation}
\end{proposition}
\begin{proof}
Write $D_\beta$ for the diagonal $V\times V$ matrix with $(\beta_i)_{i\in V}$ as diagonal coefficients, then $H_\beta=(\operatorname{Id} -PD_\beta^{-1} )D_\beta$. Since $H_\beta>0$, by Perron-Frobenius theorem, we have that $\rho(P D_\beta^{-1})<1$, where $\rho(P D_\beta^{-1})$ is the spectral radius of $P D_\beta^{-1}$. Hence, we can write the following convergent expansion,
$$
G_\beta=H_\beta^{-1}= D_\beta^{-1}\sum_{k=0}^{\infty} (P D_\beta^{-1})^k,
$$
which exactly corresponds to \eqref{eq-tilG-pathsum}.

For
the expansion of $\exp(u_\beta(i,j))$, note first that \(\sum_{\sigma\in \bar{\mathcal{P}}_{j,i}^{V}}\frac{W_{\sigma}}{\beta_{\sigma}^{-}}\leq \beta_{i}G_\beta(j,i)<\infty\).
A path in \(\mathcal{P}_{j,i}^{V}\) can be cut at its first visit to \(i\), turning it into the concatenation of a path in \(\mathcal{\bar{P}}_{j,i}^{V}\) and a path in \(\mathcal{P}_{i,i}^{V}\), and this operation is bijective. It implies that
\begin{equation}
\label{eq-factorize-path-sum}
\left(\sum_{\sigma\in \bar{\mathcal{P}}_{j,i}^{V}}\frac{W_{\sigma}}{(2\beta)_{\sigma}^-}\right)G_\beta(i,i)=\left(\sum_{\sigma\in \bar{\mathcal{P}}_{j,i}^{V}}\frac{W_{\sigma}}{(2\beta)_{\sigma}^-}\right)\left(\sum_{\sigma\in \mathcal{P}_{i,i}^V}\frac{W_{\sigma}}{(2\beta)_{\sigma}}\right)=\sum_{\sigma\in \mathcal{P}_{j,i}^V}\frac{W_{\sigma}}{(2\beta)_{\sigma}}=
G_\beta(i,j),
\end{equation}
hence the result.
\end{proof}
\subsection{A priori estimates on 
$\mathcal{Q}_{i_0}^{W}(du)$.}\label{ss_estimate_disertori}
The following proposition is borrowed from~\cite[Lemma 3]{disertori2010quasi}. For convenience, we give a shorter proof of that estimate based on
spanning trees instead of fermionic variables, following the proof of the corresponding result for the
ERRW, c.f.~\cite[Lemma 7]{disertori2014transience}.
\begin{proposition}
\label{estimate_disertori}
  Let \((\mathcal{G}=(V,E), W)\) be a finite graph with conductances. Fix a vertex \(i_{0}\).
  Let \(\eta>0\) and let $e_1=\{\underline{e_1}, \overline{e_1}\}, \ldots, e_K=\{\underline{e_K}, \overline{e_K}\}$ be $K$ distinct undirected edges such that
  \(W_{e_k}\ge 2\eta\) for all \(k=1, \ldots, K\). Then
\[
\int 
\exp\left(\eta\sum_{k=1}^{K}\cosh\left(u_{\overline{e_k}}-u_{\underline{e_k}}\right)\right)\mathcal{Q}_{i_0}^{W}(du)
\leq e^{\eta K} 2^{K/2}\]
where $\mathcal{Q}_{i_0}^{W}(du)$ is the probability distribution defined in Theorem~\ref{thm_vrjp_u_ST}.
\end{proposition}
\begin{proof}
We remind that $\mathcal{Q}_{i_0}^{W}(du)$ is defined by
\[
\mathcal{Q}_{i_0}^{W}(du)= \frac{1}{\sqrt{2\pi}^{|V|-1}}\exp(-\sum_{i}u_i-\sum_{i\sim j}W_{i,j}(\cosh(u_i-u_j)-1))\sqrt{D(W,u)} du_{V\setminus\{i_{0}\}},\]
with \(du_{V\setminus\{i_{0}\}}=\prod_{i\neq i_0} du_i\) and \(
D(W,u)=\sum_{T\in \ttt} \prod_{\{i,j\}\in T} W_{i,j}e^{u_i+u_j}
\) where the sum is on spanning trees.

Let \(\tilde{W}=W-\eta\sum_{k=1}^{K}\mathds{1}_{e_k}\), (i.e.\ \(\tilde W\) is equal to \(W-\eta\) on the edges $e_1, \ldots, e_K$, and unchanged on
the other edges). By assumption, we have \(\tilde W_{i,j}>0\) on the edges, and for all spanning trees \(T\), since edges appear at most once:
\begin{align*}
\prod_{\{i,j\}\in T}W_{i,j}e^{u_i+u_j}&\leq
  \left(\prod_{k=1}^K \frac{W_{e_k}}{W_{e_k}-\eta}\right)
\prod_{\{i,j\}\in T}\tilde W_{i,j}e^{u_i+u_j}
\le 2^K \prod_{\{i,j\}\in T}\tilde W_{i,j}e^{u_i+u_j},
\end{align*}
which implies
\(
D(W,u)\le 2^K D(\tilde W,u).
\)
From
the expression of \(\mathcal{Q}^{W}_{i_0}(du)\), we deduce that
\[
\exp\left(\eta\sum_{k=1}^{K}\cosh\left(u_{\overline{e_k}}-u_{\underline{e_{k}}}\right)\right)
\mathcal{Q}^{W}_{i_0}(du) \le e^{\eta K} 2^{K/2} \mathcal{Q}^{\tilde W}_{i_0}(du).
\]
It implies that
\begin{align*}
\int  \exp\left(\eta\sum_{k=1}^{K}\cosh\left(u_{\overline{e_k}}-u_{\underline{e_k}}\right)\right)
\mathcal{Q}^{W}_{i_0}(du)
\le e^{\eta K} 2^{K/2} \int \mathcal{Q}^{\tilde W}_{i_0}(du) = e^{\eta K} 2^{K/2}.
\end{align*}

\end{proof}

\section{The wired boundary condition and Kolmogorov extension to infinite graphs}\label{ss_wired}
\subsection{Restriction with wired boundary condition}
Our objective is to extend the relations between the VRJP and the \(\beta\) field to the case of infinite graphs.
To this end, we need an appropriate boundary condition, which turns out to be the wired boundary condition.

\begin{definition}
\label{wired_bd_def}
Let \(\mathcal{G}=(V,E)\) be a connected graph with finite degree at each site, and \(V_{1}\) a strict finite subset of \(V\). We define the restriction of \(\mathcal{G}\) to \(V_{1}\) with wired boundary condition as the graph \(\mathcal{G}_{1}=(\tilde{V}_{1}=V_{1}\cup\{\delta\},E_{1})\) where \(\delta\) is an extra point and
\[E_{1}=\{ \{i,j\}\in E,\ \text{s.t. }i\in V_{1},j\in V_{1},i\sim j\}\cup\{\{i,\delta\}, i\in V_{1} \text{ s.t. } \exists j\notin V_{1},i\sim j\}.\]
If \((W_{i,j})_{\{i,j\}\in E}\) is a set of positive conductances, we define \((W^{(1)}_{i,j})_{\{i,j\}\in E_{1}}\) as the set of restricted conductances by
\[
\begin{cases}
  W_{i,j}^{(1)}=W_{i,j}, & \text{if }i,j\in V_{1}, \;\{i,j\}\in E_1,\\
  W_{i,\delta}^{(1)}=\sum_{j\notin V_{1},j\sim i}W_{i,j}, & \text{if } \{i,\delta\}\in E_1,\\
  0,&\text{otherwise.}
\end{cases}
\]
\end{definition}

\begin{remark}
Intuitively, this restriction corresponds to identifying all points in \(V\setminus V_1\) to a single point \(\delta\) and to delete the edges connecting points of \(V\setminus V_1\).
The new weights are obtained by summing the weights of the edges identified by this procedure.
\end{remark}
The following lemma is fundamental and is the justification for the choice of this notion of restriction.
\begin{lemma}
\label{restriction}
Let \((\mathcal{G}=(V,E), W)\) be a finite graph with conductances and $\nu_V^W$ the associated distribution on random potentials defined in Theorem~\ref{thm_potential_fini}. 
Let \(V_1\) be a strict subset of \(V\) and 
\((\mathcal{G}_1=(\tilde V_1,E_1), W^{(1)})\) be the restriction of $(\ggg,W)$ to $V_1$ with wired boundary condition. Let $\nu_{\tilde V_1}^{W^{(1)}}$ be the distribution of random potential associated with $(\ggg_1, W^{(1)})$. 
We denote by $\left(\nu_V^W\right)_{|V_1}$ and $\left(\nu_{\tilde V_1}^{W^{(1)}}\right)_{|V_1}$ the marginal distributions on $V_1$ of $\nu_{V}^W$ and $\nu_{\tilde V_1}^{W^{(1)}}$.
Then
\[
\left(\nu_V^W\right)_{|V_1}=\left(\nu_{\tilde V_1}^{W^{(1)}}\right)_{|V_1}.
\]
\end{lemma}
\begin{remark} Note that there  is no such compatibility relation with the more usual notion of restriction of graph. The
wired boundary condition is fundamental and in fact will be responsible for the extra gamma random variable that
appears in the representation of the VRJP on the infinite graph.
\end{remark}
\begin{proof}
Taking $(\lambda_i)_{i\in V}\in \R_+^{V}$ such that \(\lambda_{V\setminus V_1}=0\) in Theorem~\ref{thm_potential_fini},
we get that
\begin{align*}
&
\int e^{-\sum_{i\in V_1} \lambda_i \beta_i}\nu_V^W(d\beta)
\\
=&\exp\left(-\sum_{i\sim j,i,j\in V_{1}}W_{i,j}(\sqrt{(1+\lambda_{i})(1+\lambda_{j})}-1)-\sum_{i\sim j,i\in V_{1},j\notin V_{1}}(W_{i,j}(\sqrt{1+\lambda_{i}}-1))\right)\prod_{i\in V_{1}}\frac{1}{\sqrt{1+\lambda_{i}}}.
\end{align*}
Applying Theorem~\ref{thm_potential_fini} to the graph \(\ggg_1\) with $(\lambda_i)_{i\in \tilde V_1}\in \R_+^{\tilde V_1}$ such that \(\lambda_{\delta}=0\), we get
\begin{align*}
&\int  e^{-\sum_{i\in V_1} \lambda_i \beta_i}
\nu_{\tilde V_1}^{W^{(1)}}(d\beta)
\\
=&
\exp\left(-\sum_{i\sim j,i,j\in V_{1}}W^{(1)}_{i,j}(\sqrt{(1+\lambda_{i})(1+\lambda_{j})}-1)-\sum_{i\in V_{1}, i\simGun \delta}(W_{i,\delta}^{(1)}(\sqrt{1+\lambda_{i}}-1))\right)\prod_{i\in V_{1}}\frac{1}{\sqrt{1+\lambda_{i}}}.
\end{align*}
By definition of \(W^{(1)}_{i,j}\), these Laplace transforms are equal, hence the marginal distributions are equal.
\end{proof}

\subsection{Kolmogorov extension: proof of Proposition~\ref{kolmogorov} and Definition~\ref{def-G-and-psi}}
\label{ss_Kolmogorov}
Let \(\mathcal{G}=(V,E)\) be a connected infinite graph with finite degree at each site with conductances
\((W_{i,j})\).
We remind that \((V_n)_{n\ge 1}\) is an increasing sequence of finite strict subsets of \(V\) that exhausts \(V\), i.e.
\(
\cup_{n} V_n =V.
\)

Let \(\mathcal{G}_n=(\tilde{V}_{n}=V_n\cup\{\delta_n\}, E_n)\) be the restriction of \(\mathcal{G}\) to \(V_n\) with wired boundary condition, and \((W^{(n)})\) the restricted conductances. By construction, if \(n<m\), then \((\ggg_n, W^{(n)})\) is the restriction of \((\ggg_m,W^{(m)})\) to $V_n$ with wired boundary condition.
Lemma~\ref{restriction} implies that the sequence of marginal  distributions $\left(\nu_{\tilde V_n}^{W^{(n)}}\right)_{|V_n}$ is a compatible sequence of probabilities. By Kolmogorov extension theorem, it implies that there exists a probability measure \(\nu_V^{W}\) such that
$$
\left(\nu_V^{W}\right)_{|V_n}= \left(\nu_{\tilde V_n}^{W^{(n)}}\right)_{|V_n},
$$
for all integer $n$. By Theorem~\ref{thm_potential_fini}, $\nu_V^W(\d\beta)$ is supported by the set of potentials $\beta$ such that $(H_\beta)_{V_n,V_n}$ is positive definite for all integers $n$, hence by $\ddd_V^W$. It also implies the other properties of $\nu_V^W(d\beta)$.

The solution of the equation defining $\psi^{(n)}_\beta$ in Definition~\ref{def-G-and-psi} exists and is unique since it is equivalent to $(\psi^{(n)}_\beta)_{V_n^c}=1$ and
\begin{align}
\label{existence_psi}
(H_\beta)_{V_n,V_n}(\psi^{(n)}_\beta)_{V_n}(i)=\sum_{j\sim i, j\in V_n^c} W_{i,j}, \;\;\; \hbox{ for }i\in V_n.
\end{align}
Since $(H_\beta)_{V_n,V_n}$ is positive definite for $\beta\in \ddd_V^W$, it defines $\psi^{(n)}_\beta$ uniquely.

\subsection{Coupling lemma. Definition of \(G^{(n)}\), and relations with \(\hat G^{(n)}\), \(\psi^{(n)}\) and \(\gamma\).}
\label{ss_G-et-tilde}

Consider the probability $\nu_V^W(d\beta,d\gamma)$ defined in \eqref{nu_beta_gamma}. It will be convenient to couple the measure $\nu_V^W(d\beta,d\gamma)$ and the measure $\nu_{\tilde V_n}^{W^{(n)}}(d\beta)$ in the following way. 
\begin{lemma}
\label{coupling}
For $\beta\in \ddd_V^W$ and $\gamma>0$, we define
\(\beta^{(n)}\in \R^{\widetilde V_n}\) by
\begin{equation}
\label{betad}
\beta^{(n)}_{V_n}=\beta_{V_n},\
\beta^{(n)}_{\delta_n}= \sum_{j\in V_n, j \sim \delta_n}\frac{1}{2} W^{(n)}_{j,\delta_n} \psi_\beta^{(n)}(j)+ \gamma.
\end{equation}
Then, $\beta^{(n)}\in\ddd_{\widetilde V_n}^{W^{(n)}}$ and under $\nu_V^W(d\beta,d\gamma)$, $\beta^{(n)}$ is distributed according to $\nu_{\widetilde V_n}^{W^{(n)}}$.

Let $H^{(n)}_{\beta^{(n)}}$ be the Schr\"odinger operator associated with $\ggg_n$, $W^{(n)}$ and potential $\beta^{(n)}$. Let
$
G_{\beta^{(n)}}^{(n)}=(H^{(n)}_{\beta^{(n)}})^{-1},
$
be its Green function. Then,
$$
G_{\beta^{(n)}}^{(n)}(\delta_n,\delta_n)={1\over 2\gamma},
$$
and, for all $i\in V_n$,
$$
\psi^{(n)}_\beta(i)={G_{\beta^{(n)}}^{(n)}(\delta_n,i)\over G_{\beta^{(n)}}^{(n)}(\delta_n,\delta_n)}=e^{u^{(n)}_{\beta^{(n)}}(\delta_n,i)},
$$
where \(u^{(n)}_{\beta^{(n)}}\) is the field defined in Proposition~\ref{prop_u_beta_STZ} for the graph $\ggg_n$ and with the potential $\beta^{(n)}$.

As usual, we often omit the subscript $\beta$ and write $H^{(n)}$, $G^{(n)}$, $u^{(n)}$, and consider them as random variables on $\ddd_V^W\times \R_+^*$ under $\nu_V^W(d\beta,d\gamma)$.
\end{lemma}
\begin{proof}
Let $\beta\in \ddd_V^W$ and $\gamma>0$. Denote in this proof  by $(u(j))_{j\in \tilde V_n}$ the vector defined by
$$
u(j)=
\begin{cases}0,&\hbox{ if $j=\delta_n$},
\\
\log \psi_\beta^{(n)}(j),   &\hbox{ if $j\in V_n$}.
\end{cases}
$$
Then, by definition of $\psi^{(n)}_\beta$ and $\beta^{(n)}$, we have $(H_{\beta^{(n)}}^{(n)}(e^u))_{V_n}=(H_{\beta}(\psi^{(n)}_\beta))_{V_n}=0$ and 
$$
H_{\beta^{(n)}}^{(n)}(e^u)(\delta_n)=2\beta^{(n)}_{\delta_n}- \sum_{j\in V_n, j \sim \delta_n}W^{(n)}_{\delta_n,j} \psi_\beta^{(n)}(j)=2\gamma.
$$
Since $(e^{u(j)})$ is a vector with positive coefficients, by general results on symmetric M-matrices, see Theorem~2.7 page 141 of \cite{berman1994nonnegative}, it implies that $H^{(n)}_{\beta^{(n)}}>0$. Moreover, it implies that ${1\over 2\gamma}e^{u(\cdot)}=G_{\beta^{(n)}}^{(n)}(\delta_n,\cdot)$, hence that $G_{\beta^{(n)}}^{(n)}(\delta_n,\delta_n)={1\over 2\gamma}$ and $e^{u(\cdot)}={G_{\beta^{(n)}}^{(n)}(\delta_n,\cdot)\over G_{\beta^{(n)}}^{(n)}(\delta_n,\delta_n)}=e^{u_{\beta^{(n)}}^{(n)}(\delta_n,\cdot)}$.

Finally, by Theorem~\ref{thm_u_beta_STZ}, the law of $(\beta^{(n)}_{V_n}, G^{(n)}_{\beta^{(n)}}(\delta_n,\delta_n))$ is the same under $\nu_V^W(d\beta,d\gamma)$ and $\nu_{\tilde V_n}^{W^{(n)}}(d\beta^{(n)})$, and since $\beta^{(n)}\mapsto (\beta^{(n)}_{V_n}, G^{(n)}_{\beta^{(n)}}(\delta_n,\delta_n))$ is a bijection by Proposition~\ref{prop_u_beta_STZ}, it implies that under $\nu_V^W(d\beta,d\gamma)$, $\beta^{(n)}$ has law $\nu_{\tilde V_n}^{W^{(n)}}$.
\end{proof}

\begin{proposition}\label{relate_psi_G}
With the definition of Proposition~\ref{coupling}, for all $i,j\in V_n$, all $\beta\in \ddd_V^W$, all $\gamma >0$,
 \[{G^{(n)}_{\beta^{(n)}}(i,j)}=
\hat G_\beta^{(n)}(i,j)+{1\over 2\gamma}\psi_\beta^{(n)}(i)\psi_\beta^{(n)}(j).\]
\end{proposition}
\begin{proof}
For simplicity, we omit the subscripts $\beta$, $\beta^{(n)}$ in $\hat G^{(n)}_{\beta}$, $\psi^{(n)}_\beta$, $G_{\beta^{(n)}}^{(n)}$ in the expression below.
By Proposition~\ref{fini_green}, Lemma~\ref{coupling}, using $(\beta^{(n)})_{V_n}=\beta_{V_n}$, we find that
\begin{align}\label{path_G}
  G^{(n)}(i,j)=\sum_{\sigma\in \mathcal{P}_{i,j}^{\widetilde{V}_{n}}}\frac{W^{(n)}_{\sigma}}{(2\beta^{(n)})_{\sigma}}
 ,\;\;\; 
 \hat G^{(n)}(i,j)= \sum_{\sigma\in \mathcal{P}_{i,j}^{{V}_{n}}}\frac{W_{\sigma}}{(2\beta)_{\sigma}}
\end{align}
and
\[
 \psi^{(n)}(i)=\frac{G^{(n)}(\delta_n,i)}{G^{(n)}(\delta_n,\delta_n)}= \sum_{\sigma\in \overline{\mathcal{P}}_{i,\delta_n}^{\widetilde{V}_{n}}}\frac{W^{(n)}_{\sigma}}{(2\beta)^-_{\sigma}}.
\]
Therefore, if we denote \(\mathcal{P}_{i,\delta_n,j}^{\widetilde{V}_{n}}\) the collection of paths on \(\tilde{V}_{n}\) starting from \(i\), visiting \(\delta_n\) at least once, and ending
at \(j\), that is,
\[\mathcal{P}_{i,\delta_n,j}^{\widetilde{V}_{n}}=\{\sigma=(\sigma_{0},\cdots,\sigma_{m})\in \mathcal{P}_{i,j}^{\widetilde{V}_{n}},
\hbox{ such that } \exists 0\leq k\leq m,\sigma_{k}=\delta_n\}
\]
then, since $(W^{(n)})_{V_n,V_n}=W_{V_n,V_n}$ and $(\beta^{(n)})_{V_n}=\beta_{V_n}$,
\begin{align*}
  G^{(n)}(i,j)-\hat{G}^{(n)}(i,j)&=\sum_{\sigma\in \mathcal{P}_{i,\delta_n,j}^{\widetilde{V}_{n}}}\frac{W^{(n)}_{\sigma}}{(2\beta^{(n)})_{\sigma}}\\
&=(\sum_{\sigma\in \bar{\mathcal{P}}_{i,\delta_n}^{\widetilde V_{n}}}\frac{W^{(n)}_{\sigma}}{(2\beta^{(n)})^{-}_{\sigma}})\cdot(\sum_{\sigma\in \mathcal{P}_{\delta_n,j}^{\widetilde{V}_{n}}}\frac{W^{(n)}_{\sigma}}{(2\beta^{(n)})_{\sigma}})\\
&=\psi^{(n)}(i)G^{(n)}(\delta_n,j)=\psi^{(n)}(i)\psi^{(n)}(j)G^{(n)}(\delta_n,\delta_n)= \psi^{(n)}(i)\psi^{(n)}(j){1\over 2\gamma}
\end{align*}
where we used Lemma~\ref{coupling} in the last equality.
\end{proof}

\section{The martingale property}\label{ss_martingale}

We remind that  \(\fff^{(n)}=\sigma(\beta_i, i\in V_n)\) is the sub $\sigma$-field generated by the random variables \(\beta\mapsto \beta_i\), \(i\in V_n\).
The following proposition is the key property for the main theorem.
\begin{proposition}
\label{prop-martingale-psi-n-crochet}
With the notations of Definition~\ref{def-G-and-psi}, for all \(n\in \N \), \(\psi^{(n)}\) has finite moments.
Moreover, we have, 
$\nu_V^W$-a.s.,
\begin{equation}
\label{eq-mart-prop-psi}
\E_{\nu_V^W}\left(\psi^{(n+1)}(i)| \fff^{(n)}\right)= \psi^{(n)}(i), \;\;\; \forall i\in V,
\end{equation}
and for all \(i,j\in V\),
\begin{equation}
\label{eq-mart-crochet}
\E_{\nu_V^W}\left(\psi^{(n+1)}(i)\psi^{(n+1)}(j)-\psi^{(n)}(i)\psi^{(n)}(j) | \fff^{(n)}\right)= \E_{\nu_V^W}\left(\hat G^{(n+1)}(i,j)-\hat G^{(n)}(i,j)| \fff^{(n)}\right).
\end{equation}
\end{proposition}
\begin{remark}
  In Theorem~\ref{thm_vrjp_u_ST}, by the change of variables \(\tilde{u}(\cdot)=u(\cdot)-\frac{\sum_{i\in V}u(i)}{|V|}\), the new variables \((\tilde{u}(i))_{i\in V}\) are in the space \(\{\sum_{i\in V}\tilde{u}(i)=0\}\) and the density becomes
  \[\mathcal{\tilde{Q}}_{i_{0}}^{W}(d \tilde{u})=\frac{1}{\sqrt{2}^{|V|-1}}e^{\tilde{u}(i_{0})}e^{-\sum_{i\sim j}W_{i,j}(\cosh(\tilde{u}(i)-\tilde{u}(j))-1)}\sqrt{D(W,\tilde{u})} d \tilde{u}_{V\setminus\{i_{0}\}}.\]
We see from this expression that \(e^{\tilde u(i)-\tilde u(i_0)}\cdot\mathcal{\tilde{Q}}_{i_{0}}^{W}=\mathcal{\tilde{Q}}_{i}^{W}\), hence that
\(
\int e^{\tilde u(i)-\tilde u(i_0)}\mathcal{\tilde{Q}}_{i_{0}}^{W}(d \tilde{u})=1.
\)
Applied to \(V=\tilde{V}_{n}\), \(i_{0}=\delta_{n}\), we get \(\E_{\nu_V^W}(\psi^{(n)}(i))=1\)
which is a particular case of~\eqref{eq-mart-prop-psi}.
\end{remark}
The original proof of that property was rather technical (see the second arXiv version of the present paper). Some time after the first version of this paper was posted on arXiv, 
a simpler proof of the martingale property \eqref{eq-mart-prop-psi} was given in \cite{DMR15}. Moreover, using some
supersymmetric arguments, the following more general property was proved.
\begin{lemma}[\cite{DMR15}]\label{exp-martingale}
 Let $\lambda\in (\R_+)^V$ be a non-negative function on $V$ with bounded support, then
$$
\E\left(e^{-\left<\lambda,\psi^{(n+1)}\right>-\demi\left<\lambda, \hat G^{(n+1)}\lambda\right>} | \fff^{(n)}\right)=e^{-\left<\lambda,\psi^{(n)}\right>-\demi\left<\lambda, \hat G^{(n)}\lambda\right>}.
$$
\end{lemma}
We provide here a different proof of this assertion based on elementary computations on the measures $\nu_V^W$ on finite sets. It also provides a simpler proof of the
original assertion Proposition~\ref{prop-martingale-psi-n-crochet} by differentiating in $\lambda$.
\subsection{Marginal and conditional laws of \(\nu_V^{W}\)} 

In this subsection, we suppose that \(\mathcal{G}=(V,E)\) is finite. We state some identities on marginal and conditional laws of
the distribution \(\nu_V^{W}\), which will be instrumental in the proof of the martingale property in the next subsection.

Let us first remark that the law $\nu_V^{W}$ defined in Theorem~\ref{thm_potential_fini} can be extended to the case where 
$P=(W_{i,j})_{i,j\in V}$ has non-zero, diagonal coefficients. Indeed, if some diagonal coefficients of $P$ are positive, then changing
from variables $(\beta_i)$ to variables $(\beta_i-\demi W_{i,i})$, we get the law $\nu_V^{\widetilde W}$ where $(\widetilde W_{i,j})$ is obtained from
$(W_{i,j})$ by replacing all diagonal entries by 0. While it is not very natural from the point of view of the VRJP to allow
non-zero diagonal coefficients, it is convenient in this section to allow this possibility since it simplifies the statements about conditional law.

Recall that for any function $\zeta:V\mapsto \R$ and any subset $U\subset V$, we write $\zeta_U$ for the restriction of $\zeta$ to the subset $U$. Similarly, if $A$ is a $V\times V$ matrix and $U\subset V$, $U'\subset V$, we write $A_{U,U'}$ for its restriction to the block $U\times U'$.
 We also write 
\(d\beta_{U}=\prod_{i\in U}d\beta_{i}\) to denote integration on variables $\beta_U$. 

In the next lemma we give an extension of the family of probability distributions \(\nu_V^{W}\). This extension was proposed by Letac, in the unpublished note \cite{Letac} discussing the integral defined in \cite{STZ}. We give a proof of this lemma using Theorem~\ref{thm_potential_fini}.

\begin{lemma}[Letac, \cite{Letac}]
  \label{lem_restriction}
  Let \(V\) be finite and \(P=(W_{i,j})_{i,j\in V}\) be a symmetric matrix with non-negative coefficients. 
 Let  \((\eta_i)_{i\in V}\in \R_+^V\) be a vector with non-negative coefficients.
Then the following measure on \(\ddd_V^W\)
 \begin{align}  \label{nu_eta}
   \nu_V^{W,\eta}(d\beta)
   &:=e^{-\demi\left<\eta, (H_\beta)^{-1} \eta\right>}e^{\left<\eta,1\right>}  \nu_V^{W}(d\beta)
\\
\nonumber
&=\mathds{1}_{H_{\beta}>0}\left(\frac{2}{\pi}\right)^{|V|/2}e^{-\demi\left<1, H_\beta 1\right>-\demi\left<\eta, (H_\beta)^{-1} \eta\right> }\frac{1}{\sqrt{\det H_{\beta}}} e^{\left<\eta,1\right>} d\beta_V
\end{align}
is a probability distribution, where \(1\) in the scalar products \(\left< 1,H_{\beta}1 \right>\) and \(\left< \eta, 1 \right>\)  is to be understood as the vector \(
\begin{pmatrix}
  1\\ \vdots \\1
\end{pmatrix}
\).
Its Laplace transform is, for any \(\lambda\in \mathbb{R}_{+}^{V}\)
\begin{equation}
  \label{eq:laplace-nubetaweta}
\int e^{-\left< \lambda,\beta \right>} \nu^{W,\eta}_V(d\beta)=e^{-\left< \eta,\sqrt{\lambda+1}-1 \right>-\sum_{i\sim j}W_{i,j}\left( \sqrt{(1+\lambda_{i})(1+\lambda_{j})}-1 \right)}\prod_{i\in V}\frac{1}{\sqrt{1+\lambda_{i}}}
\end{equation}
where \(\sqrt{\lambda+1}-1\) should be considered as the vector \((\sqrt{\lambda_i+1}-1)_{i\in V}\).
\end{lemma}
It appears in the following lemma that this extension describes all marginal laws of \(\nu_V^{W}\), and that the larger family $\nu_V^{W,\eta}$ is stable by taking marginals and conditional distributions.
\begin{lemma}\label{lem_marg_cond}
 Assume that $V$ is finite and let \(U\subset V\) be a subset.
 Under \(\nu_V^{W,\eta}(d\beta)\),
\begin{enumerate}[label=(\roman*)] 
\item
\label{item_restriction} 
$\beta_{U}$ is distributed according to
\(
 \nu_{U}^{W_{U,U},\hat\eta},
\)
where 
\begin{align}\label{eta}
\hat\eta= \eta_{U}+ P_{U,U^c} (1_{U^c}),
\end{align}
\item
  \label{lem_conditioning}
conditionally on \(\beta_{U}\), \(\beta_{U^c}\) is distributed according to \(\nu_{U^c}^{\widecheck W,\widecheck\eta}\),
where
\(\widecheck P=(\widecheck W_{i,j})_{i,j\in U^c}\) and \(\widecheck \eta\in(\R_+)^{U^c}\) are the matrix and vector defined by
\[
\check P=P_{U^c,U^c}+ P_{U^c,U} \left((H_\beta)_{U,U}\right)^{-1} P_{U,U^c},
\;\;\; \check \eta=\eta_{U^c}+P_{U^c,U} \left((H_\beta)_{U,U}\right)^{-1}(\eta_{U}).
\]
\end{enumerate}
\end{lemma}
\begin{remark} Note that $\check P$ has non-zero diagonal coefficients.
\end{remark}
\noindent N.B. As we can observe, all the quantities with \(\check \cdot\) are relative to vectors or matrices on \(U^{c}\), while 
the quantities with \(\hat \cdot\) are relative to vectors or matrices on \(U\).
\begin{lemma}
  \label{lem-exp-martingale}
  Let \(\mathcal{G}=(V,E)\) be a finite connected graph endowed with conductances \(P=(W_{i,j})_{i,j\in V}\).
 Let  \((\eta_i)_{i\in V}\in \R_+^V\) be a vector with non-negative coefficients.
Let \(U\subset V\).
For $\beta\in \ddd_V^W$, define \(\psi_\beta=G_{\beta}\eta\) where \(G_{\beta}=H_{\beta}^{-1}\); define \(\hat{\eta}=\eta_{U}+P_{U,U^{c}}1_{U^{c}}\), \(\hat{G}_\beta^{U}=((H_{\beta})_{U,U})^{-1}\) and \(\hat{\psi}_\beta=\hat{G}_\beta^{U}(\hat{\eta})\). For any \(\lambda\in \mathbb{R}_{+}^{V}\), we have, $\nu_V^{W,\eta}$ a.s.,
  \begin{equation}
    \label{eq:exp-martingale-lemma}
    \E_{\nu_{V}^{W,\eta}}(e^{-\left< \lambda,\psi \right>-\frac{1}{2}\left< \lambda,G \lambda \right>}\left| \mathcal{F}_{U}\right.)=e^{-\left< \lambda_{U},\hat{\psi} \right>-\left< \lambda_{U^c},1_{U^c} \right>-\frac{1}{2}\left< \lambda_{U},\hat{G}^{U}\lambda_{U} \right>}
  \end{equation}
  where \(\mathcal{F}_{U}=\sigma(\beta_{i},i\in U)\).
\end{lemma}

\begin{proof}[Proof of Lemma~\ref{lem_restriction} and Lemma~\ref{lem_marg_cond}]
Lemma~\ref{lem_restriction} and the assertions \ref{item_restriction} and \ref{lem_conditioning} of Lemma~\ref{lem_marg_cond} are consequences of the same decomposition of the measure $\nu_V^{W,\eta}$. It is partially inspired by computations in \cite{Letac}.
  We write \(H_{\beta}\) as block matrix
  \[H_{\beta}=
    \begin{pmatrix}
      H_{U,U} & -P_{U,U^{c}} \\ -P_{U^{c},U} & H_{U^{c},U^{c}} 
    \end{pmatrix} \text{ and define }\hat {G}^{U}=(H_{U,U})^{-1}.
  \]
  Now, define the Schur's complement
  \begin{align}\label{Hcheck}
  \check H^{U^{c}}=H_{U^{c},U^{c}}-P_{U^{c},U}\hat{G}^{U}P_{U,U^{c}}, 
  \end{align}
  and
  \begin{align*}
  \check G^{U^c}=\left( \check H^{U^{c}}\right)^{-1}.
  \end{align*}
 We have
  \begin{equation}
    \label{eq-hbeta-productof3}
H_{\beta}=
    \begin{pmatrix}
      I_{U} & 0 \\ -P_{U^{c},U}\hat{G}^{U} & I_{U^{c}}
    \end{pmatrix}
    \begin{pmatrix}
      H_{U,U} & 0 \\ 0 & \check H^{U^{c}}
    \end{pmatrix}
    \begin{pmatrix}
      I_{U} & -\hat{G}^{U}P_{U,U^{c}} \\ 0 & I_{U^{c}}
    \end{pmatrix}. 
\end{equation}
Remark that with notations of \ref{lem_conditioning} of Lemma~\ref{lem_marg_cond} we have
$$
\check H^{U^c}=2\beta_{U^c}-\check P.
$$
By \eqref{eq-hbeta-productof3}, we have
  \begin{equation}
    \label{eq-1hbeta1}
    \begin{aligned}
    &\left< 1,H_{\beta}1 \right>
    \\=&
    \left< 1_{U^c},\check{H}^{U^c}1_{U^c} \right>+\left< 1_{U},H_{U,U}1_{U} \right>+
    \left< 1_{U^c},P_{U^{c},U}\hat{G}^{U} P_{U,U^{c}} 1_{U^c}\right>-2\left< 1_{U},P_{U,U^{c}} 1_{U^c}\right>.
    \end{aligned}
 \end{equation}
On the other hand, by~\eqref{eq-hbeta-productof3} again, we have
 \begin{equation}
    \label{eq-schur-pour-gbeta}
    G_{\beta}= H_\beta^{-1}=
               \begin{pmatrix}
                 I_{U} & \hat{G}^{U}P_{U,U^{c}} \\ 0 & I_{U^{c}}
               \end{pmatrix}                                                           \begin{pmatrix}
\hat{G}^{U} & 0\\ 0 & \check{G}^{U^c}                                                               \end{pmatrix}
                          \begin{pmatrix}
                            I_{U} & 0 \\ P_{U^{c},U}\hat{G}^{U} & I_{U^{c}}
                          \end{pmatrix}\\
  \end{equation}
  therefore, since 
  \[ \begin{pmatrix}
                            I_{U} & 0 \\ P_{U^{c},U}\hat{G}^{U} & I_{U^{c}}
                          \end{pmatrix}
                          \begin{pmatrix}
                            \eta_{U} \\ \eta_{U^{c}}
                          \end{pmatrix}=
                          \begin{pmatrix}
                            \eta_{U}\\ \check{\eta}
                          \end{pmatrix}
,\]
we get,
  \begin{equation}
    \label{eq-1gbeta1}
\left< \eta,G_{\beta}\eta \right>=\left< \eta_{U},\hat{G}^{U}\eta_{U} \right>+\left< \check \eta, \check{G}^{U^c} \check\eta \right>.
\end{equation}
Combining \eqref{eq-1hbeta1} and \eqref{eq-1gbeta1} we have
\begin{align}\label{term_exp}
 \left< 1,H_{\beta}1 \right>+\left< \eta,G_{\beta}\eta \right>-2\left<\eta,1\right>
  =&\left< 1_{U^c},\check{H}^{U^c}1_{U^c} \right>+ \left< \check{\eta},\check{G}^{U^c} \check{\eta}  \right>-2\left<\check\eta,1_{U^c}\right>\\
\nonumber   &+\left< 1_{U},H_{U,U}1_{U} \right>+\left< \hat\eta,\hat{G}^{U}\hat\eta \right>-2\left<\hat\eta,1_{U}\right>.
\end{align}
By \eqref{eq-hbeta-productof3}, we also have
  \begin{align}\label{prod_det}
  \det H_{\beta}=\det H_{U,U} \det \check{H}^{U^c},\ \ \mathds{1}_{H_{\beta}>0}=\mathds{1}_{H_{U,U}>0}\mathds{1}_{\check{H}^{U^c}>0}.
  \end{align}
Combining \eqref{term_exp} and \eqref{prod_det}, we have,
  \begin{align}
  \nonumber
& \left( \frac{2}{\pi} \right)^{|V|/2} e^{-\frac{1}{2}\left< 1,H_{\beta}1 \right>-\frac{1}{2}\left< \eta,G_{\beta}\eta \right>+\left<\eta,1\right>}\frac{\mathds{1}_{H_{\beta}>0}}{\sqrt{\det H_{\beta}}} 
    \\
    \label{decomposition}
    &=\left( \frac{2}{\pi} \right)^{|U|/2} e^{-\frac{1}{2} \left< 1_U,H_{U,U}1_U \right>-\demi\left<\hat \eta,\hat{G}^{U}\hat\eta \right>+\left<\hat\eta,1_{U} \right>  }\frac{\mathds{1}_{H_{U,U}>0}}{\det H_{U,U}} 
    \\
     \nonumber
    &\;\;\;\;\cdot \left( \frac{2}{\pi} \right)^{|U^{c}|/2}  e^{-\frac{1}{2}\left< 1_{U^c},\check{H}^{U^c}1_{U^c} \right>
   -\demi \left< \check{\eta},\check{G}^{U^c} \check{\eta}  \right>+\left<\check\eta,1_{U^c}\right>}\frac{\mathds{1}_{ \check{H}^{U^c}>0}}{\sqrt{\det \check{H}^{U^c}}} .
  \end{align}
We remark that the left-hand side is the density of $\nu_V^{W,\eta}(d\beta)$, that the first term of the right-hand side corresponds to the density of $\nu_{U}^{W_{U,U},\hat\eta}(d\beta_U)$ and that, $\beta_U$ being fixed, the second term of the right-hand side is the density of $\nu_{U^c}^{\check W,\check\eta}(d\beta_{U^c})$. (Indeed, as remarked above, $\check H^{U^c}=2\beta_{U^c}-\check P$ and $\check P$, $\check \eta$ are $\beta_U$-measurable).
\ali
{\it Proof of Lemma~\ref{lem_restriction}.} Take $\eta=0$. Then $\check \eta=0$. Integrating on $d\beta_{U^c}$ on both sides of \eqref{decomposition}, with $\beta_{U}$ fixed, gives 
 \begin{align*}
\int \left( \frac{2}{\pi} \right)^{|V|/2} e^{-\frac{1}{2}\left< 1,H_{\beta}1 \right>}\frac{\mathds{1}_{H_{\beta}>0}}{\sqrt{\det H_{\beta}}} d\beta_{U^c}
    =\left( \frac{2}{\pi} \right)^{|U|/2} e^{-\frac{1}{2} \left< 1_U,H_{U,U}1_U \right>-\demi\left<\hat \eta,\hat{G}^{U}\hat\eta \right>+\left< 1_{U},\hat\eta \right> }\frac{\mathds{1}_{H_{U,U}>0}}{\det H_{U,U}}\\
\end{align*}
since $\int \nu_{U^c}^{\check W}(d\beta_{U^c})=1$ by Theorem~\ref{thm_potential_fini}. Integrating on $d\beta_U$, it gives $\int \nu_U^{W_{U,U},\hat\eta}(d\beta_U)=1$, since $\nu_V^W$ is a probability. Hence, $\nu_U^{W_{U,U},\hat\eta}$ is a probability. This implies Lemma~\ref{lem_restriction} since this restriction procedure allows to obtain all possible parameters of the family of measures $\nu_V^{W,\eta}$. Indeed, for $V$, $W$, $\eta$, consider $\tilde V=V\cup\{\delta\}$ the set obtained by adding an extra point, and define $(\tilde W_{i,j})_{i,j\in \tilde V}$ by $\tilde W_{V,V}=W$ and $W_{i,\delta}=\eta_i$ for $i\in V$. Then if we apply the previous identity to $\tilde V$ and $U:=V\subset \tilde V$, we get $\hat\eta=\eta$ and $\nu_{V}^{W,\eta}$ is a probability by the previous argument.
\ali
{\it Proof of Lemma~\ref{lem_marg_cond}.}
 Integrating on $d\beta_{U^c}$ on both sides of \eqref{decomposition}, with $\beta_{U}$ fixed, gives 
 \begin{align*}
&\int \left( \frac{2}{\pi} \right)^{|V|/2} e^{-\frac{1}{2}\left< 1,H_{\beta}1 \right>-\frac{1}{2}\left< \eta,G_{\beta}\eta \right>+\left<\eta,1\right>}\frac{\mathds{1}_{H_{\beta}>0}}{\sqrt{\det H_{\beta}}}\left(d\beta_{U^c}\right)
\\   
    =
    &
    \left( \frac{2}{\pi} \right)^{|U|/2} e^{-\frac{1}{2} \left< 1_U,H_{U,U}1_U \right>-\demi\left<\hat \eta,\hat{G}^{U}\hat\eta \right>+\left< 1_{U},\hat\eta \right> }\frac{\mathds{1}_{H_{U,U}>0}}{\det H_{U,U}}\\
\end{align*}
since $\int \nu_{U^c}^{\check W,\check \eta}(d\beta_{U^c})=1$ by Lemma~\ref{lem_restriction}. Hence, the marginal distribution of $\beta_{U}$ is $\nu_{U}^{W_{U,U},\hat\eta}$, proving
\ref{item_restriction}.
Finally, \ref{lem_conditioning} is a consequence of  the conditional probability density formula. Indeed, if we denote temporary by $f(\beta)$ the density of $\nu_V^W(d\beta)$, by $f_U(\beta_U)$ its marginal density on $U$ and by $f_{U^c}^{\beta_U}(\beta_{U^c})$ the conditional density of $\beta_{U^c}$ conditioned on $\beta_U$, we have
by \eqref{decomposition} and \ref{item_restriction},
$$
f_{U^c}^{\beta_U}(\beta_{U^c})=\frac{f(\beta)}{f_U(\beta_U)}=
\left( \frac{2}{\pi} \right)^{|U^{c}|/2}  e^{-\frac{1}{2}\left< 1_{U^c},\check{H}^{U^c}1_{U^c} \right>
   -\demi \left< \check{\eta},\check{G}^{U^c} \check{\eta}  \right>+\left<\check\eta,1_{U^c}\right>}\frac{\mathds{1}_{ \check{H}^{U^c}>0}}{\sqrt{\det \check{H}^{U^c}}}.
 $$
 Since $\check{H}^{U^c}=2\beta_{U^c}-\widecheck P$, $\check{G}^{U^c}=(\check{H}^{U^c})^{-1}$ and $\widecheck P$, $\widecheck\eta$ are $\beta_U$-measurable, it implies that the right-hand side is the density of $\nu_{U^c}^{\check W,\check \eta}(d\beta_{U^c})$.
\end{proof}

\begin{proof}[Proof of Lemma~\ref{lem-exp-martingale}]
We take the same notations as in the proof of Lemma~\ref{lem_marg_cond}.
By Lemma~\ref{lem_marg_cond}, under $\nu_V^W(d\beta)$, the law of \(\beta_{U^{c}}\), conditionally on \(\beta_{U}\), is \(\nu_{U^{c}}^{\check{W},\check{\eta}}\). Now, we set
\[
\check{\psi}=\check{G}^{U^c}\check{\eta}.
\]
By \eqref{eq-schur-pour-gbeta}, we have
  \begin{align*}
    &\left< \lambda,\psi \right>+\frac{1}{2}\left< \lambda,G\lambda \right>=\left< \lambda,G\eta \right>+\frac{1}{2}\left< \lambda,G\lambda \right>\\
    &=\left( \lambda_{U},\lambda_{U}\hat{G}^{U}P_{U,U^{c}}+\lambda_{U^{c}} \right)
      \begin{pmatrix}
        \hat{G}^{U} & 0 \\ 0 & \check{G}^{U^c}
      \end{pmatrix}
                                         \begin{pmatrix}
                                           \eta_{U}\\ P_{U^{c},U}\hat{G}^{U}\eta_{U}+\eta_{U^{c}}
                                         \end{pmatrix}\\
    &+ \frac{1}{2}\left( \lambda_{U}, \lambda_{U}\hat{G}^{U}P_{U,U^{c}}+\lambda_{U^{c}} \right) \begin{pmatrix}
        \hat{G}^{U} & 0 \\ 0 & \check{G}^{U^c}
      \end{pmatrix}
                                         \begin{pmatrix}
                                           \lambda_{U}\\ P_{U^{c},U}\hat{G}^{U}\lambda_{U}+\lambda_{U^{c}}.
                                         \end{pmatrix}
  \end{align*}
If we define \(\check{\lambda}=\lambda_{U^{c}}+P_{U^{c},U}\hat{G}^{U}\lambda_{U}\in \mathbb{R}_{+}^{U^{c}}\), we have
  \begin{align*}
    \left< \lambda,\psi \right>+\frac{1}{2}\left< \lambda,G\lambda \right>&=\left< \check{\lambda},\check{\psi} \right>+\frac{1}{2}\left< \check{\lambda},\check{G}^{U^c}\check{\lambda} \right>+\left< \lambda_{U},\hat G^{U} \eta_{U}  \right>+\frac{1}{2}\left< \lambda_{U},\hat{G}^{U}\lambda_{U} \right>
    \\
&=\left< \check{\lambda},\check{\psi} \right>+\frac{1}{2}\left< \check{\lambda},\check{G}^{U^c}\check{\lambda} \right>+\left< \lambda_{U},\hat{\psi} \right>+\frac{1}{2}\left< \lambda_{U},\hat{G}^{U}\lambda_{U} \right>-\left< 1_{U^{c}},\check{\lambda}-\lambda_{U^{c}} \right>.
  \end{align*}
  Now, remark that
  $$
  \left< \check{\lambda},\check{\psi} \right>+\frac{1}{2}\left< \check{\lambda},\check{G}^{U^c}\check{\lambda} \right>+
  \frac{1}{2}\left< \check{\eta},\check{G}^{U^c}\check{\eta} \right>=
  \frac{1}{2}\left< \check{\lambda}+\check{\eta},\check{G}^{U^c}(\check{\lambda}+\check{\eta}) \right>.
  $$
We get,
\begin{align*}
 \E_{\nu_V^{W,\eta}}\left(e^{-\left< \lambda,\psi \right>-\frac{1}{2}\left< \lambda,G\lambda \right>}\left| \mathcal{F}_{U}\right.\right)
 &=e^{-\left< \lambda_{U},\hat{\psi} \right>-\left< \lambda_{U^{c}},1_{U^{c}} \right>-\frac{1}{2}\left< \lambda_{U},\hat{G}^{U}\lambda_{U} \right>}
 \E_{\nu_{U^c}^{\check W,\check \eta}}\left(e^{-\left< \check{\lambda},\check{\psi} \right>-\frac{1}{2}\left< \check{\lambda},\check{G}^{U^c}\check{\lambda} \right>
 +\left< 1_{U^{c}},\check{\lambda} \right>}\right)\\
 &=
 e^{-\left< \lambda_{U},\hat{\psi} \right>-\left< \lambda_{U^{c}} , 1_{U^{c}}\right>-\frac{1}{2}\left< \lambda_{U},\hat{G}^{U}\lambda_{U} \right>}
  \E_{\nu_{U^c}^{\check W,\check\lambda+\check \eta}}\left(1 \right)
\end{align*}
which concludes the proof of the lemma, using that $\nu_{U^c}^{\check W, \check \lambda+\check \eta}$ is a probability
\end{proof}
\subsection{Proof of Lemma~\ref{exp-martingale}}
Remark that since $\psi^{(n)}$ is defined for all $n$ by
$$
\begin{cases}
(H_\beta \psi^{(n)})_{V_n}=0,
\\
\psi^{(n)}_{V_n^c}=1,
\end{cases}
$$  
we have $\psi^{(n)}_{V_n}=((H_\beta)_{V_n, V_n})^{-1}(\eta^{(n)})$, where
$
\eta^{(n)}=P_{V_n,V_n^c}(1_{V_n^c}).
$
Moreover, by Lemma~\ref{lem_marg_cond} (i), under $\nu_V^W(d\beta)$, we know that $\beta_{V_{n}}$ has law $\nu_{V_n}^{W,\eta^{(n)}}$. Using
Lemma~\ref{lem-exp-martingale} applied to $V=V_{n+1}$ and $U=V_n$, we have that $\hat{G}^{(n+1)}_{V_{n+1},V_{n+1}}$ corresponds to $G_\beta$ in Lemma~\ref{lem-exp-martingale}
and $\hat{G}^{(n)}_{V_{n},V_{n}}$ to $\hat G^U$, $\eta^{(n+1)}$ to $\eta$, and $\eta^{(n)}$ to $\hat \eta$. Hence, we get that a.s.
\begin{align*}
\E_{\nu_V^W}\left(e^{-\left<\lambda_{V_{n+1}},\psi^{(n+1)}_{V_{n+1}}\right>-\demi\left<\lambda_{V_{n+1}}, \hat G^{(n+1)}\lambda_{V_{n+1}}\right>} \left| \fff^{(n)}\right.\right)&=
e^{-\left<\lambda_{V_n},\psi^{(n)}_{V_n}\right>-\left<\lambda_{V_{n+1}\setminus V_n},1_{V_{n+1}\setminus V_n}\right>
-\demi\left<\lambda_{V_{n}}, \hat G^{(n)}\lambda_{V_{n}}\right>}
\\
&=
e^{-\left<\lambda_{V_{n+1}},\psi^{(n)}_{V_{n+1}}\right>-\demi\left<\lambda_{V_n}, \hat G^{(n)}\lambda_{V_n}\right>}
\end{align*}
since $\psi^{(n)}_{V_{n+1}\setminus V_n}=1$. This concludes the proof since $\psi^{(n)}$ and $\psi^{(n+1)}$ both equal 1 on $V_{n+1}^c$.

\section{Passing to the limit: proof of Theorem~\ref{convergence}, Proposition~\ref{coro-escape-proba}, Proposition~\ref{0-1law}}\label{ss_thm1}

\subsection{Proof of Theorem~\ref{convergence} i) and ii)}
\label{sec:repr-sums-paths}
\begin{proof}[Proof of Theorem~\ref{convergence} i) and ii)]
By the path representation \eqref{path_G}, we know that $\hat G_\beta^{(n)}(i,j)$ is non-decreasing for all $i,j\in V$ since the set of paths $\ppp_{i,j}^{V_n}$ is non-decreasing, and that $\hat G_\beta^{(n)}(i,j)\le G^{(n)}(i,j)$ since $\ppp_{i,j}^{V_n}\subset \ppp_{i,j}^{\widetilde V_n}$. Hence, it convergences a.s. to a random variable $\hat G(i,j)$. Since $\hat G^{(n)}(i,j)>0$ as soon as $i,j\in V_n$ (indeed, $V_n$ is connected), we have $\hat G(i,j)>0$ a.s.. It remains to prove that $\hat G(i,j)<\infty$.
As \(\hat{G}^{(n)}(i,i)\) converges a.s.\ to \(\hat{G}(i,i)\) and is non-decreasing,  
for any \(h\geq 0\),
\begin{align*}
  \nu_V^W(\hat{G}(i,i)\leq h)&=\nu_V^W(\lim_{n\to \infty}\hat{G}^{(n)}(i,i)\leq h)\\
&=\lim_{n\to \infty}\nu_V^W(\hat{G}^{(n)}(i,i)\leq h)\\
&\geq \lim_{n\to \infty}\nu_V^W(G^{(n)}(i,i)\leq h)
\\
&=\P\left({1\over 2\Gamma}\leq h\right) ,
\end{align*}
where $\Gamma$ is a gamma random variable with parameters $(\demi,1)$. In the last equality, we used that \(\frac{1}{2{G}^{(n)}(i,i)}\) has gamma law with parameters $(\demi,1)$ by Theorem~\ref{thm_u_beta_STZ}.
Therefore, \(\hat{G}(i,i)<\infty\) a.s.
For the off diagonal term, 
since \((H_{\beta})_{V_{n},V_n}\) is positive definite, we have by Cauchy-Schwarz inequality
\[\hat G^{(n)}(i,j)=\left<\delta_i,\hat G^{(n)}\delta_j\right>\leq \sqrt{\left<\delta_i,\hat G^{(n)}\delta_i\right>\left<\delta_j,\hat G^{(n)}\delta_j\right>}=\sqrt{\hat G^{(n)}(i,i)\hat G^{(n)}(j,j)}\]
therefore, \(\hat G(i,j)\leq \sqrt{\hat G(i,i)\hat G(j,j)}\) and \(\hat{G}(i,j)\) is a.s.\ finite.

From Proposition~\ref{prop-martingale-psi-n-crochet}, we know that \(\psi^{(n)}(i)\) is a positive martingale for all \(i\in V\).
As a positive martingale, \(\psi^{(n)}(i)\) converges a.s.\ to some non-negative integrable random variable \(\psi(i)\).

It remains to show that the limits $\psi$ and $\hat G$ do not depend on the choice of the exhausting sequence \((V_{n})\). Assume that \((\Omega_{n})\) is another increasing exhausting sequence, we can similarly construct the martingale \(\phi^{(n)}(i)\) associated with \(\Omega_{n}\). 
As \((\Omega_{n})\) and \((V_{n})\) are exhausting,
we can construct a subsequence \(n_k\) such that the alternating sequence \(V_{n_1},\Omega_{n_2}, V_{n_3}, \ldots\) is increasing and thus the alternating sequence \(\psi^{(n_1)}(i), \phi^{(n_2)}(i), \psi^{(n_3)}(i), \ldots\) is a martingale for all \(i\in V\). This martingale converges a.s.\ and this identifies
the limits of \(\psi^{(n)}(i)\) and \(\phi^{(n)}(i)\). The argument is the same for $\hat G$ since the sequence of Green functions associated with the alternating sequence of subsets is non-decreasing and converges a.s.
\end{proof}
\subsection{Representation of the VRJP as a mixture  on the infinite graphs: proof of iii)}
\label{sec:repr-infin-graph}
With the coupling of Section~\ref{ss_G-et-tilde}, by Proposition~\ref{relate_psi_G}, we have for $\beta\in \ddd_V^W$ and $\gamma >0$, 
 \[{G^{(n)}(i,j)}=\hat G^{(n)}(i,j)+{1\over 2\gamma} \psi^{(n)}(i)\psi^{(n)}(j).\]
By Theorem~\ref{convergence}~\ref{convergence_i} and~\ref{convergence_ii}, we have that $\nu_V^W(d\beta,d\gamma)$-a.s.
\begin{equation}
\label{limit}
\lim_{n\to \infty} G^{(n)}(i,j)=G(i,j),
\end{equation}
where \(G(i,j)\) is defined in Theorem~\ref{convergence}~\ref{convergence_iii}.

The next corollary of Proposition~\ref{estimate_disertori} gives the necessary uniform integrability on jump rates
to extend the representation of the VRJP for finite graphs to infinite graphs.
\begin{corollary}
\label{coro-ui-tilG}
   For any \(i,j\in V\), there exists \(n_{0}\in \mathbb{N}\), such that the family of random variables \(\left\{\frac{G^{(n)}(i_{0},j)}{G^{(n)}(i_{0},i)}\right\}_{n\geq n_{0}}\) is uniformly integrable under $\nu_V^W(d\beta,d\gamma)$.
\end{corollary}
\begin{proof}
Choose \(n_{0}\) such that \(i,j\in V_{n_0}\), and \(i\) and \(j\) are connected by a path in \(V_{n_0}\). Denote by \(K\) the distance between \(i\) and \(j\) for the graph distance
in \(V_{n_0}\) and let $(\sigma_0=i,\sigma_1, \ldots, \sigma_K=j)$ be a directed path from  $i$ to $j$ in $V_{n_0}$. Note that it is also a directed path in any $V_n$ for $n\ge n_0$ since $V_n$ is increasing. Let 
$$\eta:=\demi\min_{k=0, \ldots, K-1}(W_{\sigma_k,\sigma_{k+1}})>0.
$$ 
Let $c(\eta)>0$ be a positive constant depending only on $\eta$ such that $e^{2x}\le c(\eta)e^{\eta \cosh(x)}$ for all reals $x$ (which exists since $2\vert x\vert\le \eta\cosh(x)$ for $x$ large enough and since $\cosh(x)\ge 1$ for all $x$). We can write with Notation \eqref{u(i,j)}, for $n\ge n_0$,
\begin{align*}
 \left(\frac{G^{(n)}(i_{0},j)}{G^{(n)}(i_{0},i)}\right)^{2}=e^{2(u^{(n)}(i_0,j)-u^{(n)}(i_0,i))}&=\prod_{k=0}^{K-1} e^{2(u^{(n)}(i_0,\sigma_{k+1})-u^{(n)}(i_0,\sigma_k))}
 \\
 &\le c(\eta)^{K}\prod_{k=0}^{K-1} e^{\eta\cosh((u^{(n)}(i_0,\sigma_{k+1})-u^{(n)}(i_0,\sigma_k))}.
\end{align*}
By Theorem~\ref{thm_u_beta_STZ} and Lemma~\ref{coupling}, under $\nu_V^W(d\beta,d\gamma)$,  $u^{(n)}(i_0,\cdot)$ has law ${\mathcal{Q}}_{i_0}^W$. Proposition~\ref{estimate_disertori} then implies that 
for \(n\geq n_{0}\),
\[\E_{\nu_V^W}\left(\left( \frac{G^{(n)}(i_{0},j)}{G^{(n)}(i_{0},i)}\right)^{2}\right)\leq 
e^{\eta K}2^{K/2}c(\eta)^K.\]
The family is uniformly bounded in \(L^{2}\), in particular uniformly integrable.
\end{proof}

Consider now a connected finite subset \(\Lambda\subset V\) containing \(i_0\) and set
\[
\partial^+\Lambda=\{j\in \Lambda^c, \; \exists i\in \Lambda \hbox{ such that \(i\sim j\) }\}.
\]
 Consider also a real $t_0>0$. Let \(T\) be the following stopping time
\[
T=t_0\wedge \inf\{t\ge 0,\ Z_{t}\notin \Lambda \}.
\]
By construction, the law of the VRJP in exchangeable time scale on \(\mathcal{G}\) up to time \(T\) equals the law of the VRJP in exchangeable time scale on \(\mathcal{G}_{n}\) up to time \(T\),
for all \(n\) such that \(\Lambda\cup \partial^+\Lambda \subset V_n\). For convenience, in this proof we write $\P_{i_0}^{\operatorname{VRJP}, \ggg}$ for its law on $\ggg$ and $\P_{i_0}^{\operatorname{VRJP}, \ggg_n}$ for its law on $\ggg_n$. Hence, our previous discussion formally means that $\P_{i_0}^{\operatorname{VRJP}, \ggg}\left((Z_t)_{t\le T}\in \cdot\right)= \P_{i_0}^{\operatorname{VRJP}, \ggg_n}\left((Z_t)_{t\le T}\in \cdot\right)$, for $n$ large enough.  We denote by
\[
\ell_i(T)=\int_0^T \mathds{1}_{Z_u=i}\; du
\]
the local time of \(Z\) up to time \(T\). Using 
Corollary~\ref{rep_VRJP_fini} and the coupling in 
Lemma~\ref{coupling},
the VRJP in exchangeable time scale on \(\ggg_n\), starting at \(i_0\),  is a mixture of Markov jump processes with jump rates from \(i\) to \(j\)
\begin{eqnarray}\label{jumping-n}
\demi W^{(n)}_{i,j} {G^{(n)}(i_0,j)\over G^{(n)}(i_0,i)}
\end{eqnarray}
under the law $\nu_V^W(d\beta,d\gamma)$. We denote by
\[
\tilde \beta^{(n)}_i=\sum_{j\sim i} \demi W^{(n)}_{i,j} {G^{(n)}(i_0,j)\over G^{(n)}(i_0,i)}
\]
the holding time at site \(i\) (note that $\tilde\beta^{(n)}_i=\beta^{(n)}_i$ for $i\neq i_0$). We denote by \(P_{i_0}^{\operatorname{MJP}}\) the law of the Markov jump process
with jump rates \(\demi W_{i,j}\) starting at \(i_0\). The Radon-Nykodim derivative of the
law of \((Z_t)_{t\le T}\) under 
the law of the Markov jump process with jump rates (\ref{jumping-n}) with respect to its law under \(P^{\operatorname{MJP}}_{i_0}\) is
\begin{equation}
\label{RN-Gn}
e^{-\sum_{i\in \Lambda} \ell_i(T)( \tilde\beta^{(n)}_i-\demi W_i)} {G^{(n)}(i_0,Z_T)\over G^{(n)}(i_0,i_0)},
\end{equation}
where as usual \(W_i=\sum_{j\sim i} W_{i,j}\). We postpone the proof of this formula to the end this subsection. 

Formula \eqref{RN-Gn} implies that for all positive
bounded test functions \(F\), for $n$ large enough,
\begin{align}
\nonumber &\E_{i_0}^{\operatorname{VRJP}, \ggg}\left( F((Z_t)_{t\le T})\right)=\E_{i_0}^{\operatorname{VRJP}, \ggg_n}\left( F((Z_t)_{t\le T})\right)
\\
\label{cv-integrals}
=&
\int \sum_{{ j\in \Lambda\cup \partial^+\Lambda}} E^{\operatorname{MJP}}_{i_0}\left( \mathds{1}_{Z_T=j} F((Z_t)_{t\le T})
e^{-\sum_{i\in \Lambda} \ell_i(T) (\tilde\beta^{(n)}_i-\demi W_i)} {G^{(n)}(i_0,j)\over G^{(n)}(i_0,i_0)}
\right)\nu_V^{W}(d\beta,d\gamma).
\end{align}
From (\ref{limit}), we have a.s.
\[
\lim_{n\to\infty} \tilde\beta^{(n)}_i=\tilde\beta_i:= \sum_{j\sim i} \demi W_{i,j} {G(i_0,j)\over G(i_0,i)}.
\]
Remark that in \eqref{cv-integrals}, the term $e^{-\sum_{i\in \Lambda} \ell_i(T) (\tilde\beta^{(n)}_i-\demi W_i)}$ is bounded since $\Lambda\cup\partial_+\Lambda$ is finite
and $T\le t_0$.
Using the uniform integrability of \(\frac{G^{(n)}(i_0,j)}{G^{(n)}(i_0,i_0)}\), Corollary~\ref{coro-ui-tilG}, we get, letting $n$ tend to $\infty$, that
\begin{align*}
&\E_{i_0}^{\operatorname{VRJP}, \ggg}\left( F((Z_t)_{t\le T})\right)
\\
=&
\int \sum_{ {j\in \Lambda\cup\partial^+\Lambda}} E^{\operatorname{MJP}}_{i_0}\left(\mathds{1}_{Z_T=j} F((Z_t)_{t\le T})
e^{-\sum_{i\in \Lambda} \demi \ell_i(T) ( \tilde\beta_i-\demi W_i)}
 {G(i_0,j)\over G(i_0,i_0)}\right)
\nu_V^{W}(d\beta,d\gamma)
\\
=&
\int
E^{\beta,\gamma,i_0}_{i_0}
\left(F((Z_t)_{t\le T})
\right)
\nu_V^{W}(d\beta,d\gamma)
\end{align*}
where \(E^{\beta,\gamma,i_0}_{i_0}\) is the expectation associated with the probability \(P^{\beta,\gamma,i_0}_{i_0}\) defined in Theorem~\ref{convergence}.
Since $\Lambda$ and $t_0$ can be chosen arbitrarily, the previous identity characterizes the law of $(Z(t))_{t\ge 0}$. This concludes the proof of Theorem~\ref{convergence}~\ref{convergence_iii}.
\begin{proof}[Proof of formula \eqref{RN-Gn}]
Consider the Markov jump process on the graph $\ggg_n$ with jump rates $\demi W^{(n)}$, denote by $P^{MJP,(n)}_{i_0}$ its law starting from $i_0$. At each vertex $i$, its waits an exponential random time with parameter $\demi W^{(n)}_i$, then jumps to $j\sim i$ with probability proportional to $W^{(n)}_{i,j}$.
On time interval $[0,t]$, 
the probability that it 
follows the discrete path $(\sigma_0=i_0,\sigma_1, \ldots, \sigma_n)$ and jumps at  times $0<s_1<\ldots <s_n<t$ has distribution 
\begin{align*}
&
\left(\prod_{k=0}^{n-1} {W^{(n)}_{\sigma_k,\sigma_{k+1}}\over W^{(n)}_{\sigma_k}}\right) \left(e^{-\demi W^{(n)}_{\sigma_n}(t-s_n)} \prod_{k=0}^{n-1} W^{(n)}_{\sigma_k}e^{-\sum_{k=0}^{n-1}\demi W^{(n)}_{\sigma_k} (s_{k+1}-s_k)}\right)ds_1\cdots ds_n
\\
&=
\left(\prod_{k=0}^{n-1} W^{(n)}_{\sigma_k,\sigma_{k+1}}\right)e^{-\demi\sum_{i\in \tilde V_n} W^{(n)}_i \ell_i((\sigma_k),(s_k))} ds_1\cdots ds_n,
\end{align*}
where $\ell_i((\sigma_k),(s_k))$ is the total time spent at position $i$ by the trajectory with discrete path $(\sigma_k)$ and jump times $(s_k)$. The same formula is true for
the Markov jump process with jump rates \eqref{jumping-n} with $\demi W^{(n)}_{i,j}$ replaced by $\demi W^{(n)}_{i,j} {G^{(n)}(i_0,j)\over G^{(n)}(i_0,i)}$. By cancellation of the ratios  ${G^{(n)}(i_0,j)\over G^{(n)}(i_0,i)}$ along the trajectory, it gives that on time interval $[0,t]$, the Radon-Nikodym derivative of the law of the Markov jump process with jump rates \eqref{jumping-n} starting at $i_0$ with respect to $P^{MJP,(n)}_{i_0}$ is
$$
M_t:=e^{-\sum_{i\in \tilde V_n} \ell_i(t)( \tilde\beta^{(n)}_i-\demi W^{(n)}_i)} {G^{(n)}(i_0,Z_t)\over G^{(n)}(i_0,i_0)}.
$$
Moreover, $M_t$ is a martingale and is bounded on finite time intervals. Since $T$ is a bounded stopping time, $T\le t_0$, we have \eqref{RN-Gn} for the stopping time $T$ since $W^{(n)}_{V_n,V_n}=W_{V_n,V_n}$.
\end{proof}

\subsection{Proof of Proposition~\ref{coro-escape-proba}, and iv) of Theorem~\ref{convergence}}
\label{sec:equiv-betw-psi=0}

\begin{proof}[Proof of Proposition~\ref{coro-escape-proba}]
 Recall notation of Section~\ref{ss_sum_paths} and identities~\eqref{eq-tilG-pathsum}. As \(n\mapsto \hat{G}^{(n)}(i,j)\) is increasing,
we have
\begin{align}\label{path-hatG}
\hat{G}(i,j)=\sum_{\sigma\in \mathcal{P}_{i,j}^{V}}\frac{W_{\sigma}}{(2\beta)_{\sigma}}.
\end{align}
By arguments similar to~(\ref{eq-factorize-path-sum}), we have
\[
\frac{\hat{G}(i_{0},i)}{\hat{G}(i_{0},i_{0})}=\sum_{\sigma\in \bar{\mathcal{P}}_{i,i_{0}}^{V}}\frac{W_{\sigma}}{(2\beta)_{\sigma}^{-}}.
\]
We recall that $(\tilde Z_n)_{n\in \N}$ denotes the discrete time process which describes successive jumps of the process $(Z_t)_{t\in \R_+}$. Clearly, $\{\tau^+_{i_0}<\infty\}=\{\exists n\ge 1, \hbox{ s.t. } \tilde Z_n=i_0\}$.
Therefore, if we denote \(\{(\tilde{Z}_{n})\sim \sigma\}=\{\tilde{Z}_{0}=\sigma_{0},\ldots,\tilde{Z}_{m}=\sigma_{m}\}\) with \(m=|\sigma|\), then for \(i\neq i_{0}\)
\begin{equation}
\label{eq-h(i)-neqi0}
\begin{aligned}
  h(i)&:=P_{i}^{\beta,\gamma,i_{0}}(\tau^+_{i_{0}}<\infty)=\sum_{\sigma\in \bar{\mathcal{P}}_{i,i_{0}}^{V}}P_{i}^{\beta,\gamma,i_{0}}((\tilde{Z}_{n})\sim \sigma)
  \\&
  =\sum_{\sigma\in \bar{\mathcal{P}}_{i,i_{0}}^{V}}\frac{W_{\sigma}}{(2\beta)_{\sigma}^{-}}\frac{G(i_{0},i_{0})}{G(i_{0},i)}=\frac{\hat{G}(i_{0},i)}{\hat{G}(i_{0},i_{0})}\cdot \frac{{G}(i_{0},i_{0})}{G(i_{0},i)}.
\end{aligned}
\end{equation}
It follows from \(G(i,j)=\hat{G}(i,j)+\frac{1}{2\gamma}\psi(i)\psi(j)\) that, for \(i\neq i_{0}\),
\begin{align*}
  P_{i}^{\beta,\gamma,i_{0}}(\tau^+_{i_{0}}=\infty)&=1-h(i)
  =\frac{\psi(i_{0})}{2\gamma}\frac{\hat{G}(i_{0},i_{0})\psi(i)-\hat{G}(i_{0},i)\psi(i_{0})}{\hat{G}(i_{0},i_{0})G(i_{0},i)}.
\end{align*}
Therefore,
\begin{align}
\nonumber
  P_{i_{0}}^{\beta,\gamma,i_{0}}(\tau_{i_{0}}^{+}=\infty)&=\sum_{j\sim i_{0}}\frac{W_{i_{0},j}G(i_{0},j)}{2 \tilde{\beta}_{i_{0}}G(i_{0},i_{0})}P_{j}^{\beta,\gamma,i_{0}}(\tau^+_{i_{0}}=\infty)\\
\label{expression-max-principal}  
&=\sum_{j\sim i_{0}}\frac{\psi(i_{0})W_{i_{0},j}}{4\gamma \tilde{\beta}_{i_{0}}}\frac{\hat{G}(i_{0},i_{0})\psi(j)-\hat{G}(i_{0},j)\psi(i_{0}) }{\hat{G}(i_{0},i_{0})G(i_{0},i_{0}) }.
\end{align}
By definition, for $n$ large enough, we have \(H \hat{G}^{(n)}(i_{0},\cdot)=\mathds{1}_{i_{0}}(\cdot)\). Taking the limit \(n\to \infty\),
 we have \(H \hat{G}(i_{0},\cdot)=\mathds{1}_{i_{0}}(\cdot)\). By (iii) of Theorem~\ref{main} (proved in section \ref{s_Proof-Thm2}),
 we have \(H \psi(\cdot)=0\), therefore,
\[\sum_{j\sim i_{0}}W_{i_{0},j}[\psi(j)\hat{G}(i_{0},i_{0})-\psi(i_{0})\hat{G}(i_{0},j)]=\psi(i_{0}),\]
hence
\( P_{i_{0}}^{\beta,\gamma,i_{0}}(\tau_{i_{0}}^{+}=\infty)=\frac{\psi(i_{0})^{2}}{4\gamma \tilde{\beta}_{i_{0}}\hat{G}(i_{0},i_{0})G(i_{0},i_{0})}.\)
\end{proof}
\begin{proof}[Proof of Theorem~\ref{convergence}, (iv)]
From Proposition~\ref{coro-escape-proba}, we see that \(P_{i_{0}}^{\beta,\gamma,i_{0}}(\tau_{i_{0}}^{+}=\infty)>0\) if and only if
\(\psi(i_0)>0\). Since the Markov jump process \(P_{i_{0}}^{\beta,\gamma,i_{0}}\) is irreducible (\(\ggg\) is connected), (iv) follows.
\end{proof}
\subsection{Ergodicity and the 0-1 law: proof of Proposition~\ref{0-1law} and~\ref{0-1law-ERRW}}
\begin{proof}[Proof of Proposition~\ref{0-1law}.]
From the expression of the Laplace transform of \(\beta\), c.f. Proposition~\ref{kolmogorov},
we see that under $\nu_V^W(d\beta)$, \((\beta_i)_{i\in V}\) is stationary for the action of \(\aaa\).

By 1-dependence, c.f. Proposition~\ref{kolmogorov}, it is also ergodic.
Indeed, assume that \((\tau_n)\in \aaa^\N\) is a sequence of automorphims such that \(d_\ggg(i_0, \tau_n(i_0))\to \infty\)
for some vertex \(i_0\). We prove that \((\tau_n)\) is mixing in the sense that for all \(A,B\in \sigma(\beta_i, i\in V)\)
\begin{equation}\label{mixing}
\lim_{n\to \infty} \nu_V^W (\tau_{n}^{-1}(B)\cap A)= \nu_V^W(A)\nu_V^W(B).
\end{equation}
Assume that \(V_1\subset V\)
is finite and that \(A,B\in \sigma(\beta_j, \;j\in V_1)\).
By 1-dependence, \(\tau^{-1}_n(B)\) is independent of \(A\) for \(n\) large enough. This implies that \eqref{mixing} is true for all $A,B$ in the algebra
$\ppp= \cup_{V_1 \hbox{ finite}}\sigma(\beta_j, \;j\in V_1)$.
Now, any measurable set in $\sigma\{\beta_j, \; j\in V\}$ can be approximated by an element in $\ppp$, i.e. for $\epsilon>0$ and $A,B$ in  $\sigma\{\beta_j, \; j\in V\}$ we can find $A_0, B_0\in \ppp$ such that
$\nu_V^W(A\Delta A_0)<\epsilon$ and $\nu_V^W(B\Delta B_0)<\epsilon$ where $\Delta$ is the symmetric difference (see e.g. \cite[Theorem~D, Section~13, page 56]{Halmos}). Hence, also $\nu_V^W(\tau_n^{-1}( B)\Delta \tau_n^{-1} (B_0))<\epsilon$ since $\tau_n$ is measure preserving. This easily imply \eqref{mixing} for all $A$ and $B$.
Finally, if $A$ is $\tau$-invariant, \eqref{mixing} implies $(\nu_V^W(A))^2=\nu_V^W(A)$, hence $\nu_V^W(A)$ equals $0$~or~$1$.

Let us prove stationarity and ergodicity of the random variables $(\hat G(i,j))_{i,j\in V}$. We also denote by $\tau$ the transformation on $\R^{V\times V}$ given by,
for $(M(i,j))\in \R^{V\times V}$, $\tau M(i,j)= M(\tau i, \tau j)$. Since the limit $\hat G_\beta (i,j)$ does not depend on the choice of the sequence $V_n$, a.s., we have that 
$\tau(\hat G_\beta)=\hat G_{\tau(\beta)}$, by choosing the approximating sequences $V_n$ and $\tau V_n$.
It implies that $(\hat G(i,j))$ is stationary. Moreover, if $A\in \bbb(\R^{V\times V})$ is $\tau$-invariant, then the set $\{\beta, \; \hat G_\beta \in A\}$ is a.s. $\tau$-invariant. Hence, $A$ has measure 0 or 1 under the law of $\hat G$.   The proof is similar for the random variables $(\psi(i))_{i\in V}$.

The event \(\{\psi(i)=0, \;\forall i\in V\}\) is clearly invariant by \(\aaa\), hence has probability 0 or 1. Together with (iv) of Theorem~\ref{convergence}
it concludes the proof of the proposition.
\end{proof}
\begin{proof}[Proof of Proposition~\ref{0-1law-ERRW}.]
The proof is similar to the proof of Proposition~\ref{0-1law}.
\end{proof}

\subsection{Proof of Theorem~\ref{main}: relation with spectral properties of the random schr\"o\-din\-ger operator}
\label{sec:aspect-rand-schr}

\begin{proof}[Proof of Theorem~\ref{main} (i)]\label{s_Proof-Thm2}
By Proposition~\ref{kolmogorov}, since $\nu_V^W$ is supported on $\ddd_V^W$, we have a.s.  that \(H_{V_n\times V_n}>0\) and passing to the limit, we get
\(H\ge 0\). Hence, \(\sigma(H)\subset [0,+\infty)\).
\end{proof}
\begin{proof}[Proof of Theorem~\ref{main} (ii)]
  As \(-\varepsilon\) is strictly outside the spectrum of \(H \) a.s., the equation \((H +\varepsilon)\hat{G}^{\varepsilon}=\Id\) has a unique finite solution, we can verify that
\(\sum_{\sigma\in \mathcal{P}^{V}_{i,j}}\frac{W_{\sigma}}{(2\beta+\varepsilon)_{\sigma}}\) is a solution to this equation. Now by \eqref{path-hatG} 
we have
  \[(H+\varepsilon)^{-1}(i,j)=\sum_{\sigma\in \mathcal{P}^{V}_{i,j}}\frac{W_{\sigma}}{(2\beta+\varepsilon)_{\sigma}}\leq \sum_{\sigma\in \mathcal{P}^{V}_{i,j}}\frac{W_{\sigma}}{(2\beta)_{\sigma}}=
  \hat{G}(i,j)<\infty.\]
Therefore, as \(\sum_{\sigma\in \mathcal{P}_{i,j}}\frac{W_{\sigma}}{(2\beta+\varepsilon)_\sigma}\) is increasing as \(\varepsilon\to 0\), it converges a.s.\ to \(\hat{G}(i,j)\). 
\end{proof}
\begin{proof}[Proof of Theorem~\ref{main}~(iii)]
  We have, for all \(i\in V_n\), $\beta\in \ddd_V^W$,
\(\psi^{(n)}_\beta(i)=\sum_{j\sim i}\frac{W_{i,j}}{2\beta_{i}}\psi_\beta^{(n)}(j).\)
As \(\psi^{(n)}(i)\) converges a.s.\ to \(\psi(i)\), the above equality holds in the limit, i.e., for all \(i\in V\), a.s.
\[\psi(i)=\sum_{j\sim i}\frac{W_{i,j}}{2\beta_{i}}\psi(j),\]
this exactly means \((H\psi)(i)=0\).
\end{proof}
\begin{proof}[Proof of Theorem~\ref{main}~(iv)]
By Fatou's Lemma, the limit \(\psi(i)\) satisfies \(\E_{\nu_V^W} (\psi(i))\leq 1\). By Markov inequality
\[\nu_V^W\left(\psi(i)\geq C\| i\|_\infty^{p}\right)\leq \frac{1}{C\| i\|_\infty^{p}}.\]
Let \(\partial B(0,n)\) be the sphere of radius \(n\) for $\| \cdot\|_\infty$, i.e. \(\partial B(0,n)=\{i\in \mathbb{Z}^{d},\; \| i \|_\infty=n\}\). When \(p> d\),
\begin{align*}
\sum_{i\in \Z^d, \; i\neq 0} \nu_V^W(\psi(i)\geq C\| i\|_\infty^{p}) &= \sum_{n\ge 1}  \sum_{i\in \partial B(0,n) }\nu_V^W(\psi(i)\geq C\| i\|_\infty^{p})
\\
&\leq \sum_{n\ge 1}  \sum_{i\in \partial B(0,n)}\frac{1}{C\| i\|_\infty^{p}}\\
&\le C' \sum_{n}\frac{n^{d-1}}{n^{p}}<\infty
\end{align*}
for some constant \(C'>0\).
By Borel-Cantelli lemma, a.s.\ only a finite number of \(i\) satisfies \(\psi(i)\geq C\Vert i\Vert_\infty^{p}\).
\end{proof}

\section{\(h\)-transforms}\label{ss_h-transform}
\begin{corollary}
\label{coro-doob-h}
Recall that \(\tau_{i_0}^+=\inf\{t\ge 0, \;\; Z_t=i_0, \; \exists s<t \hbox{ s.t. } Z_s\neq i_0\}\) is the first return time to \(i_0\) of
the process \((Z_t)_{t\ge 0}\).
\begin{enumerate}[label=(\roman*),leftmargin=*]
\item 
\label{Doob-i}
For almost all $\beta$ and $i_0\in V$, denote by \(\hat{P}_{i_{0}}^{\beta,i_{0}}\) the law of the Markov jump process with jump rate from \(i\) to \(j\)
\[
\begin{cases}
  \frac{1}{2}W_{i,j}\frac{\hat{G}(i_{0},j)}{\hat{G}(i_{0},i)}, & i\neq i_{0},\\
\tilde\beta_{i_0} \frac{W_{i_{0},j}\hat{G}(i_{0},j)}{\sum_{k\sim i_{0}}W_{i_{0},k}\hat{G}(i_{0},k)}, & i=i_{0},\ j\sim i_{0},
\end{cases}
\]
where as before \(\tilde\beta_{i_0}=\sum_{j\sim i_0} \demi W_{i_0,j} \frac{G(i_0,j)}{G(i_0,i_0)}\).
Then, for $\nu_V^W$-almost all $\beta$, for $\gamma>0$,
$$
{P}_{i_{0}}^{\beta,\gamma,i_{0}}\left((Z_t)_{t\le \tau_{i_0}^+}\in \cdot \; |\; \tau_{i_0}^+<\infty \right)
=
\hat{P}_{i_{0}}^{\beta,i_{0}}\left((Z_t)_{t\le \tau_{i_0}^+}\in \cdot \right).
$$
\item
\label{Doob-ii} 
The VRJP in exchangeable time scale, conditionally on \(\{\tau_{i_{0}}^{+}<\infty\}\) and up to its first return time to \(i_0\),
 is given by the following mixture:
\[\P_{i_{0}}^{\text{VRJP}}(\ (Z_t)_{t\le \tau^+_{i_0}}\in \cdot\ |\tau_{i_{0}}^{+}<\infty)=\int \hat{P}_{i_{0}}^{\beta,i_{0}}((Z_t)_{t\le \tau^+_{i_0}}\in \cdot) \frac{P_{i_{0}}^{\beta,\gamma,i_{0}}(\tau_{i_{0}}^{+}<\infty)}{\P_{i_{0}}^{\text{VRJP}}(\tau_{i_{0}}^{+}<\infty)} \nu_V^{W}(d\beta,d\gamma).\]
\item 
\label{Doob-iii}  Let $i_0\in V$. A.s. on the event $\{\psi(i)>0, \; \forall i\in V\}$, \(\check{G}(i_0,j):=\hat{G}(i_0,i_0)\psi(j)-\hat{G}(i_0,j)\psi(i_0)\) is positive for all $j\neq i_0$,
and we define \(\check{P}_{i_{0}}^{\beta,\gamma,i_{0}}\) as the law of the Markov jump process starting at $i_0$ and with jump rate from \(i\) to \(j\)
\begin{align*}
\begin{cases}
\frac{1}{2}W_{i,j}\frac{\check{G}(i_{0},j)}{\check{G}(i_{0},i)}, & i\neq i_{0},\ j\neq i_{0}, \\
\tilde\beta_{i_0} \frac{W_{i_{0},j}\check{G}(i_{0},j)}{\sum_{k\sim i_{0}}W_{i_{0},k}\check{G}(i_{0},k)}, & i=i_{0},\ j\sim i_{0},\\
0, & i\sim i_{0}\ j=i_{0}.
\end{cases}
\end{align*}
Then, $\nu_V^W$-almost surely on this event, for $\gamma>0$, 
$$
{P}_{i_{0}}^{\beta,\gamma,i_{0}}\left((Z_t)_{t\ge 0}\in \cdot \; |\; \tau_{i_0}^+=\infty \right)
=
\check{P}_{i_{0}}^{\beta,\gamma,i_{0}} \left((Z_t)_{t\ge 0}\in \cdot \right).
$$
\item
\label{Doob-iv} 
The VRJP in exchangeable time scale, conditionally on the event \(\{\tau_{i_{0}}^{+}=\infty\}\), is a mixture of Markov jump processes with mixing law
\begin{align*}
  \P_{i_{0}}^{\text{VRJP}}(\ \cdot \ |\tau_{i_{0}}^{+}=\infty)=\int \check{P}_{i_{0}}^{\beta,\gamma,i_{0}}(\cdot)\frac{P_{i_{0}}^{\beta,\gamma,i_{0}}(\tau_{i_{0}}^{+}=\infty)}{\P_{i_{0}}^{\text{VRJP}}(\tau_{i_{0}}^{+}=\infty)} \nu_{V}^{W}(d\beta,d\gamma).
\end{align*}
\end{enumerate}
\end{corollary}
\noindent
\begin{remark}
  Note that in the case (i), the conditional jump rates do not depend on \(\gamma\).
\end{remark}

\begin{proof}[Proof of Corollary~\ref{coro-doob-h}]
\begin{enumerate}[label=(\roman*),wide, labelwidth=!, labelindent=0pt]
\item
Recall from~(\ref{eq-h(i)-neqi0}) that for \(i\neq i_{0}\)
\begin{align*}
  h(i)&=P_{i}^{\beta,\gamma,i_{0}}(\tau^+_{i_{0}}<\infty)=\frac{\hat{G}(i_{0},i)G(i_{0},i_{0})}{\hat{G}(i_{0},i_{0})G(i_{0},i)}.
\end{align*}
For \(i\neq i_{0}\), we have
\[
P_{i_0}^{\beta,\gamma,i_{0}}(\ X_{t+dt}=j\ |X_t=i, \ t\le \tau_{i_{0}}^{+}<\infty)\sim {h(j)\over h(i)} P_{i_0}^{\beta,\gamma,i_{0}}(\ X_{t+dt}=j\ |X_t=i).
\]
Hence,  the jump rate of \(P_{i_0}^{\beta,\gamma,i_{0}}(\ \cdot\ |\tau_{i_{0}}^{+}<\infty)\), up to time \(\tau_{i_0}^+\), from \(i\) to \(j\) is
\begin{align*}
 \frac{1}{2}W_{i,j}\frac{G(i_{0},j)}{G(i_{0},i)}\frac{h(j)}{h(i)}&=\frac{1}{2}W_{i,j}\frac{\hat{G}(i_{0},j)}{\hat{G}(i_{0},i)}.
\end{align*}
The jump rate of \(P_{i_0}^{\beta,\gamma,i_{0}}(\ \cdot\ |\tau_{i_{0}}^{+}<\infty)\), up to time \(\tau_{i_0}^+\),  from \(i_{0}\) to \(j\) is given by
\begin{eqnarray*}
\demi W_{i_0,j}{G(i_0,j)\over G(i_0,i_0)}
{h(j)\over P_{i_0}^{\beta,\gamma,i_{0}}(\tau_{i_{0}}^{+}<\infty)}=\tilde \beta_{i_0}  \frac{W_{i_{0},j}\hat{G}(i_{0},j)}{\sum_{k\sim i_{0}}W_{i_{0},k}\hat{G}(i_{0},k)},
\end{eqnarray*}
where \(\tilde \beta_{i_0}=\sum_{l\sim i_0} \demi W_{i_0,l}{G(i_0,l)\over G(i_0,i_0)} \).
\item
From~\ref{Doob-i}, we have
   \begin{align*}
 &   \P_{i_{0}}^{\text{VRJP}}(\ (Z_t)_{t\le \tau^+_{i_0}}\in \cdot\ |\tau_{i_{0}}^{+}<\infty)\\
&=\int P_{i_{0}}^{\beta,\gamma,i_{0}}(\ (Z_t)_{t\le \tau^+_{i_0}}\in \cdot\ | \tau_{i_{0}}^{+}<\infty) \frac{P_{i_{0}}^{\beta,\gamma,i_{0}}(\tau_{i_{0}}^{+}<\infty)}{\P_{i_{0}}^{\text{VRJP}}(\tau_{i_{0}}^{+}<\infty)}\nu_V^{W}(d\beta,d\gamma)
\\
&=\int \hat{P}_{i_{0}}^{\beta,i_{0}}((Z_t)_{t\le \tau^+_{i_0}}\in \cdot) \frac{P_{i_{0}}^{\beta,\gamma,i_{0}}(\tau_{i_{0}}^{+}<\infty) }{\P_{i_{0}}^{\operatorname{VRJP}}(\tau_{i_{0}}^{+}<\infty)}\nu_V^{W}(d\beta,d\gamma).
  \end{align*}
\item
 The fact that \(\check{G}(i_0,j)\) is positive for $j\neq i_0$ when $\psi>0$ is a consequence of Proposition~\ref{coro-escape-proba}, and Theorem~\ref{convergence}~\ref{iv}.
Similarly to~\ref{Doob-i}, for \(i\neq i_0\), we have
\[
P_{i_0}^{\beta,\gamma,i_{0}}(\ X_{t+dt}=j\ |X_t=i, \  \tau_{i_{0}}^+=\infty)\sim {1-h(j)\over 1-h(i)} P^{\beta,\gamma,i_{0}}(\ X_{t+dt}=j\ |X_t=i).
\]
Hence,  the jump rate of \(P_{i_0}^{\beta,\gamma,i_{0}}(\ \cdot\ |\tau_{i_{0}}^{+}=\infty)\), from \(i\neq i_0\) to \(j\) is
\begin{align*}
 \frac{1}{2}W_{i,j}\frac{G(i_{0},j)}{G(i_{0},i)}\frac{1-h(j)}{1-h(i)}=
 \frac{1}{2}W_{i,j}\frac{\hat{G}(i_{0},i_{0})\psi(j)-\hat{G}(i_{0},j)\psi(i_{0})}{\hat{G}(i_{0},i_{0})\psi(i)-\hat{G}(i_{0},i)\psi(i_{0})}=\frac{1}{2}W_{i,j}\frac{\check{G}(i_{0},j)}{\check{G}(i_{0},i)}.
\end{align*}
The jump rate of \(P_{i_0}^{\beta,\gamma,i_{0}}(\ \cdot\ |\tau_{i_{0}}^{+}=\infty)\),  from \(i_{0}\) to \(j\) is given by
\begin{eqnarray*}
\demi W_{i_0,j}{G(i_0,j)\over G(i_0,i_0)}
{1-h(j)\over P_{i_0}^{\beta,\gamma,i_{0}}(\tau_{i_{0}}^{+}=\infty)}=\tilde \beta_{i_0}  \frac{W_{i_{0},j}\check{G}(i_{0},j)}{\sum_{k\sim i_{0}}W_{i_{0},k}\check{G}(i_{0},k)},
\end{eqnarray*}
where \(\tilde \beta_{i_0}=\sum_{l\sim i_0} \demi W_{i_0,l}{G(i_0,l)\over G(i_0,i_0)} \).
\item follows easily from~\ref{Doob-iii} in the same way as in~\ref{Doob-ii}.
\end{enumerate}
\end{proof}
\section{Proof of recurrence of 2-dimensional ERRW: Theorem~\ref{rec-ERRW}}\label{ss_proof-rec}
Consider the square grid \(\ggg=(\Z^2, E)\) with constant edge weight \(a_e=a>0\).
From \eqref{rep-ERRW} in Section~\ref{results-ERRW}, we know that the ERRW on \(\Z^2\) is a mixture of reversible Markov chains with conductances
\begin{equation}\label{x_rec}
x_{i,j}=W_{i,j} G(0,i)G(0,j)
\end{equation}
where $(W,\beta,\gamma)$ are distributed according $\tilde\nu_V^a(dW,d\beta,d\gamma)$.
We will use \cite{merklbounding} to prove the following lemma.
\begin{lemma}
\label{estimee-xij}
There exists \(c(a)>0\) and \(\xi(a)>0\), depending only on \(a\), such that for $\ell\in \Z^2$,
\begin{equation}
\label{poly-decrease}
\E_{\tilde\nu_V^a}\left( \left({x_\ell \over x_0}\right)^{{1\over 4}} \right)\le c(a) \| \ell\|_{\infty}^{-\xi(a)},
\end{equation}
where \(x_i=\sum_{j\sim i} x_{i,j}\) and $(x_{i,j})$ is defined in \eqref{x_rec}.
\end{lemma}
\begin{proof}
This estimate follows from Theorem~2.8 of~\cite{merklbounding} (it can also be deduced from \cite[Lemma 2.5]{merkl2009recurrence})
which gives a similar estimate on finite boxes. In~\cite[Theorem~2.8]{merklbounding}, the estimate is stated for
a periodic torus, but it is clear in the proof that the only necessary ingredient is that the finite graph with conductances is invariant by the reflection exchanging
0 and \(\ell\).
For this reason we choose the approximating sequence
 \(V_n=B({\ell\over 2}, n)\cap \Z^2\), where \(B({\ell\over 2}, n)\)  is the ball with center \(\ell/2\) and radius \(n\). Consider as in Section~\ref{ss_Kolmogorov}
the graph
\(
\ggg_{n}=(\tilde V_n=V_n\cup\{\delta_n\}, E_n),
\)
and the associated weights \((a^{(n)}_e)_{e\in E_n}\) obtained by restriction of \((\ggg,(a_e)_{e\in E})\) to \(V_n\) with wired boundary condition.
Clearly, central symmetry  with respect to \({\ell\over 2}\) (mapping \(\delta_{n}\) to itself) leaves  \((\ggg_{n}, a^{(n)})\) invariant and exchanges
\(0\) and \(\ell\).

With the coupling defined in Section~\ref{ss_G-et-tilde}, we define for \(i\sim j\), \(i,j\) in \(\tilde V_n\),
\[
x^{(n)}_{i,j}=W^{(n)}_{i,j} G^{(n)}(0,i)G^{(n)}(0,j).
\]
where \(W^{(n)}\) is obtained by restriction with wired boundary condition from \(W\).
By additivity of Gamma random variables, under $\tilde\nu_V^a$, \((W^{(n)}_e)_{e\in E_n}\)
are independent Gamma random variables with parameters \((a^{(n)}_e)_{e\in E_n}\).
Hence, the ERRW on \(\ggg_n\), with initial weights \(a^{(n)}\), starting from 0, is a mixture of reversible Markov chains with conductances \((x^{(n)}_e)_{e\in E_n}\).

From Theorem~\ref{convergence}, with the coupling defined in Section~\ref{ss_G-et-tilde}, we have that for all \(i,j\in \Z^2\), \(i\sim j\), a.s.
\begin{eqnarray}\label{xn-x}
\lim_{n\to \infty} x^{(n)}_{i,j}=x_{i,j}.
\end{eqnarray}
The proof of Theorem~2.8 of~\cite{merklbounding},
can be readily adapted to prove the following estimate. 
\begin{lemma}\label{estimée_MR}
There exists \(c(a)>0\) and \(\xi(a)>0\) only depending on \(a\) such that
for $\ell\in \Z^2$ and \(n\) large enough,
\[
\E_{\tilde\nu_V^a}\left( \left({x_\ell^{(n)}\over x_0^{(n)}}\right)^{{1\over 4}} \right)\le c(a)\| \ell\|_{\infty}^{-\xi(a)},
\]
where, with the usual convention, $x_{\ell}^{(n)}=\sum_{j\sim \ell} x_{\ell,j}^{(n)}$.
\end{lemma}
Then, Lemma~\ref{estimee-xij} follows from Lemma~\ref{estimée_MR}, \eqref{xn-x} and Fatou's lemma.
\end{proof}
We now deduce recurrence of the ERRW from the estimate \eqref{poly-decrease} and from Theorem \ref{convergence} and Proposition \ref{0-1law-ERRW}.
We have, for \(\ell\neq 0\),
\[
x_\ell=\sum_{j\sim \ell} W_{\ell,j} G(0,\ell)G(0,j)= 2\beta_\ell G(0,\ell)^2\ge {\beta_\ell\over 2\gamma^2} \psi(0)^2 \psi(\ell)^2.
\]
Similarly,
\[
x_0=\sum_{j\sim 0} W_{0,j} G(0,0)G(0,j)= G(0,0)(2\beta_0 G(0,0)-1).
\]
Hence,
\begin{eqnarray}\label{xl_x0}
{x_\ell\over x_0}\ge {\psi(0)^2\over 2\gamma^2 G(0,0)(2\beta_0 G(0,0)-1)} \beta_\ell\psi(\ell)^2.
\end{eqnarray}
Assume the ERRW is transient. By Proposition~\ref{0-1law-ERRW} it implies that, a.s.,  \(\psi(i)>0\) for all \(i\).
Choose first \(\eta>0\) such that
\[
\tilde\nu_V^a\left(
{\psi(0)^2\over 2\gamma^2 G(0,0)(2\beta_0 G(0,0)-1)} \le \eta\right)\le \demi.
\]
For all \(\epsilon>0\), we have by (\ref{poly-decrease})
\begin{eqnarray}\label{poly-epsilon}
\tilde\nu_V^a\left( {x_\ell \over x_0}\ge \epsilon \right)\le {1\over \epsilon^{{1\over 4}}} c(a)\| \ell\|_{\infty}^{-\xi(a)}.
\end{eqnarray}
On the other hand, we have by (\ref{xl_x0})
\begin{eqnarray}
\nonumber
\tilde\nu_V^a\left( {x_\ell \over x_0}\ge \epsilon \right)&\ge&
\tilde\nu_V^a\left(
{\psi(0)^2\over 2\gamma^2 G(0,0)(2\beta_0 G(0,0)-1)} >\eta, \;\; \beta_{\ell}\psi(\ell)^2> {\epsilon\over \eta} \right)
\\
\nonumber
&=&
1-
\tilde\nu_V^a\left(
\left\{{\psi(0)^2\over 2\gamma^2 G(0,0)(2\beta_0 G(0,0)-1)} \le\eta\right\}\cup\left\{\beta_{\ell}\psi(\ell)^2\le {\epsilon\over \eta}\right\} \right)
\\
\label{xl_psil}
&\ge&
 \demi- \tilde\nu_V^a \left(
\beta_{\ell}\psi(\ell)^2\le {\epsilon\over \eta}\right).
\end{eqnarray}
By Proposition~\ref{0-1law-ERRW}, \(\beta_\ell\psi(\ell)^2\) is stationary with respect to translations. Together with \eqref{xl_psil} and \eqref{poly-epsilon}, it implies that
\[
\tilde\nu_V^a\left(
\beta_0\psi(0)^2\le {\epsilon\over \eta}\right)=
\tilde\nu_V^a\left(
\beta_{\ell}\psi(\ell)^2\le {\epsilon\over \eta}\right)
\ge \demi-{1\over \epsilon^{{1\over 4}}} c(a)\| \ell\|_{\infty}^{-\xi(a)}.
\]
By sending \(\ell \) to infinity, we get
\( \tilde\nu_V^a \left(
\beta_0\psi(0)^2\le {\epsilon\over \eta}\right)\ge \demi\).
Letting $\epsilon\to 0$, this is incompatible with \(\psi(0)>0\) a.s., hence with transience of ERRW.

\section{Proof of Functional central limit theorems for the VRJP and the ERRW: Theorem~\ref{thm-clt} and~\ref{FCLT-ERRW}}\label{ss_clt}
\begin{proof}[Proof of Theorem~\ref{thm-clt} and Theorem~\ref{FCLT-ERRW}]
Let us start by the VRJP on $\Z^d$, $d\ge 3$, with constant weights \(W_{i,j}=W\). Assume that the VRJP is transient.

Recall that \((X_n)_{n\in \N}\) is the canonical discrete process on \((\Z^d)^\N\).
For $\nu_V^W$-almost all \(\beta\), let us define \(\tilde P_x^{\psi}\) to be the law of the reversible
Markov chain, starting at $x$, with conductances \(W_{i,j} \psi(i)\psi(j)\), i.e.\ with transition probabilities
\[\tilde P_x^{\psi}(X_{n+1}=j|X_{n}=i)=
\frac{W_{i,j}\psi(j)}{\sum_{l\sim i} W_{i,l} \psi(l)}.\]
Denote by \(\tilde P_x^{\beta,\gamma,0}\) the law of the underlying discrete time process associated with the Markov jump process \(P_x^{\beta,\gamma,0}\), so that
for \(i\sim j\)
\[\tilde P_x^{\beta,\gamma,0}(X_{n+1}=j|X_{n}=i)=
\frac{W_{i,j}G(0,j)}{\sum_{l\sim i} W_{i,l} G(0,l)}.\]
As \(\psi\) is a generalized eigenfunction of \(H_{\beta}\), for any \(i\in V\),
\[
\beta_i=\sum_{j\sim i}\frac{1}{2}W_{i,j}\frac{\psi(j)}{\psi(i)}.
\]
It then follows by Proposition \ref{fini_green} that, for \(i\neq 0\),
\begin{align*}
  h^\psi(i)&:=\tilde P_{i}^{\psi}(\tau^+_{0}<\infty)=\sum_{\sigma\in \bar{\mathcal{P}}_{i,0}^{V}}\tilde P_{i}^{\psi}(X_{n}\sim \sigma)
  =\sum_{\sigma\in \bar{\mathcal{P}}_{i,0}^{V}}\frac{W_{\sigma}}{(2\beta)_{\sigma}^{-}}\frac{\psi(0)}{\psi(i)}=\frac{\hat{G}(0,i)}{\hat{G}(0,0)}\frac{\psi(0)}{\psi(i)}.
\end{align*}
(recall that \(\bar{\mathcal{P}}_{i,0}^{V}\) is defined in Section~\ref{ss_sum_paths}.)
Consider the Markov chain \(\tilde P_{0}^{\psi}(\ \cdot\ |\tau_{0}^{+}=\infty)\) (Doob's \((1-h^\psi)\)-transform).
By similar computation as in the proof of Proposition~\ref{coro-doob-h}, we have that
the transition probability of \(\tilde P_{0}^{\psi}(\ \cdot\ |\tau_{0}^{+}=\infty)\) from \(i\) to \(j\) is
\begin{align*}
  \frac{W_{i,j}\psi(j)(1-h^\psi(j))}{\sum_{l\sim i} W_{i,l}\psi(l)(1-h^\psi(l))}=\frac{W_{i,j}\check{G}(0,j)}{\sum_{l\sim i}W_{i,l}\check{G}(0,l)}
\end{align*}
for \(j\neq 0\),
and \(0\) when \(j=0\).
Therefore, we see that the transition probabilities of \(\tilde P_{0}^{\psi}(\ \cdot \ |\tau_{0}^{+}=\infty)\) are the same as those of
\(\tilde{P}_{0}^{\beta,\gamma,0}(\ \cdot\ \, |\, \tau_{0}^{+}=\infty )\), c.f. iii) of Proposition~\ref{coro-doob-h}. Moreover, if we denote
\[
\xi_{0}=\sup\{n;\ X_{n}=0\},
\]
then, by strong Markov property
\[\tilde P^{\psi}_{0}(X_{n}\in \cdot|\tau_{0}^{+}=\infty)=\tilde P_{0}^{\psi}((X\circ \theta_{\xi_{0}})_{n}\in \cdot)\]
\[\tilde P^{\beta,\gamma,0}_{0}(X_{n}\in \cdot|\tau_{0}^{+}=\infty)=\tilde P_{0}^{\beta,\gamma,0}((X\circ \theta_{\xi_{0}})_{n}\in\cdot)\]
where \(\theta_{n}\) is the shift in time by \(n\).
It follows that \((X\circ \theta_{\xi_{0}})_{n}\) has the same law under \(\tilde P^{\psi}_{0}\) and under \(\tilde P^{\beta,\gamma,0}_{0}\).

Remark also, from Proposition~\ref{0-1law}, that \(
W_{i,j}\psi(i)\psi(j)\) are stationary and ergodic conductances under $\nu_V^W(d\beta)$.
We can thus apply Theorem~4.5 and Theorem4.6 of~\cite{de1989invariance}.
In order to have a functional central limit theorem we need to show that, c.f. Theorem~4.5 of~\cite{de1989invariance},
\begin{eqnarray}\label{E-conductances}
\E_{\nu_V^W}(W_{i,j}\psi(i)\psi(j))<\infty.
\end{eqnarray}
In order to show that it has non-degenerate asymptotic covariance we need to show that, c.f. Theorem~4.6 and identity~(4.20) of~\cite{de1989invariance},
\begin{eqnarray}\label{E-conductances-inverse}
\E_{\nu_V^W}\left(\frac{1}{W_{i,j}\psi(i)\psi(j)}\right)<\infty.
\end{eqnarray}
By invariance of the law of the conductances by symmetries of \(\Z^d\), we know that the limit diffusion matrix is of the form
\(\sigma^2 \Id\).

The same reasoning works in the case of the ERRW with constant weights \(a_{i,j}=a\): in this case \((W_{i,j})\) are i.i.d., but as shown in Proposition~\ref{0-1law-ERRW}, \(W_{i,j} \psi(i)\psi(j)\) is also
stationary and ergodic under \(\tilde\nu_V^{a}(dW,d\beta)\).

Estimates (\ref{E-conductances}) and (\ref{E-conductances-inverse})
are provided by \cite{disertori2010quasi} in the VRJP case, and by \cite{disertori2014transience} in the ERRW case.
This is summarized in the following lemma.
\begin{lemma}\label{estimates-psi}
(i) (VRJP case) Consider the VRJP on \(\Z^d\), for \(d\ge 3\), with constant weights \(W_{i j}=W\)
There exists \(0<\lambda_2 <\infty\) such that for \(W>\lambda_2\), the VRJP is transient and such that
(\ref{E-conductances}), (\ref{E-conductances-inverse}) are true under \(\nu_V^W(d\beta)\).

(ii) (ERRW case) Consider the ERRW on \(\Z^d\), for \(d\ge 3\), with constant weights \(a_{i j}=a\)
There exists \(0<\tilde \lambda_2 <\infty\) such that for \(a>\tilde\lambda_2\), the ERRW is transient and
(\ref{E-conductances}), (\ref{E-conductances-inverse}) are true under \(\tilde\nu_V^a(dW,d\beta)\).
\end{lemma}
The proof of that lemma is given below. We first apply it to prove the functional central limit theorem.
Consider the VRJP case. 
Assume that the condition of the lemma is satisfied. Define
\[
X^{(n)}_t=\frac{X_{\lfloor nt\rfloor}}{\sqrt n}.
\]
From~\cite{de1989invariance}, we know that there exists \(0<\sigma^2<\infty\) such that for all bounded Lipschitz function
\(F\) for the Skorokhod topology, for all \(\epsilon>0\), for all \(0<T<\infty\),
\begin{eqnarray}\label{conv-Q}
\lim_{n\to \infty}
{Q^*}\left( \left\vert \tilde E_0^\psi(F((X^{(n)}_{0\le t\le T}))-\E(F((B_{0\le t\le T}))\right\vert\ge \epsilon\right)=0,
\end{eqnarray}
where \(B_t\) is a \(d\)-dimensional Brownian motion with covariance \(\sigma^2\hbox{Id}\), and where
\(Q^*\) is the invariant measure for the processes viewed from the particle
\[
Q^*(d\beta)={\sum_{j\sim 0} W_{0,j} \psi(0)\psi(j)\over \E_{\nu_V^W}(\sum_{j\sim 0} W_{0,j} \psi(0)\psi(j))} \cdot \nu_V^W(d\beta).
\]
It is clear, since \(Q^*\) and \(\nu_V^W\) are equivalent probability distributions that (\ref{conv-Q}) is also true when
\({Q^*}\) is replaced by \({\nu_V^W}\).
This implies an annealed functional central limit theorem for the process \((X_n)\) under the annealed law \(\E_{\nu_V^W}\left(\tilde P_0^\psi(\cdot)\right)\):
\begin{eqnarray}
\label{conv-Q-bis}
\lim_{n\to \infty}
\left\vert \E_{\nu_V^W}\left(\tilde E_0^\psi(F((X^{(n)}_{0\le t\le T})\right)-\E\left(F((B_{0\le t\le T})\right)\right\vert=0.
\end{eqnarray}
Let \(\Upsilon^{(n)}_{t}:=\frac{1}{\sqrt{n}}(X\circ \theta_{\xi_{0}})_{[nt]}\).
Denote \(d^{\circ}\) the Skorohod metric on \(D([0,\infty),\mathbb{R}^{d})\),  the space of càdlag functions \(f:[0,\infty)\to \mathbb{R}^{d}\).
As
\[|{X}^{(n)}_{t}-\Upsilon^{(n)}_{t}|=\frac{1}{\sqrt{n}}|X_{[nt]}-X_{[nt+\xi_{0}]}|\leq \frac{|\xi_{0}|}{\sqrt{n}}\xrightarrow[n\to\infty]{} 0,\]
 we have
\begin{eqnarray}\label{estimate-Ups}
d^{\circ}(X^{(n)},\Upsilon^{(n)})\to 0.
\end{eqnarray}
Recall that \(F\) is a bounded Lipschitz function for the Skorohod topology, therefore,
\[|F(X_{t}^{(n)})-F(\Upsilon^{(n)}_{t})|\to 0\]
and~\eqref{conv-Q-bis} is valid for \(X^{(n)}\) replaced by \(\Upsilon^{(n)}\).
But  \(\Upsilon^{(n)}\) has the same law under \(\tilde P_0^\psi\) and \(\tilde P_0^{\beta,\gamma,0}\).
This implies the functional central limit theorem (\ref{conv-Q-bis}), for the law \(\E_{\nu_V^W}\left(\tilde P_0^{\beta,\gamma,0}(\cdot)\right)\)  in place of \(\E_{\nu_V^W}\left(\tilde P_0^\psi(\cdot)\right)\) starting from 0.
By Theorem~\ref{convergence}, the law  \(\E_{\nu_V^W}\left(\tilde P^{\beta,\gamma,0}_0(\cdot)\right)\) is that of the discrete time process $(\tilde Z_n)$ under $\P^{VRJP}_0$.

The proof is exactly the same for the ERRW, one just needs to replace the law \(\nu_V^W(d\beta)\) by the law \(\tilde \nu_V^a(dW,d\beta)\).
\end{proof}
\begin{proof}[Proof of Lemma~\ref{estimates-psi}]
Let us start by the ERRW case, ii).
Consider the sequence of subsets of \(\Z^d\), \(V_n=[-n,n]^{d}\).
Recall that
\[
\psi^{(n)}(j)=e^{u^{(n)}(\delta_n,j)},
\]
when \(j\in V_n\).
Consider the point \(y_n=(-n,0, \ldots,0)\), so that \(y_n\) is at the boundary of the set, \(y_n\sim\delta_n\).
By~\cite[Lemma 7]{disertori2014transience} (which is the ERRW's counterpart of Proposition~\ref{estimate_disertori}, Section~\ref{ss_estimate_disertori}), we have for \(a>16\),
\begin{eqnarray}\label{cosh-voisin}
\E_{\tilde \nu_V^a}\left( (\cosh(u(\delta_n,y_n))^8\right)\le 2.
\end{eqnarray}
(Indeed, the proof does not depend on the graph structure, nor on the choice of the rooting.)

From, \cite[Theorem 4]{disertori2014transience}, there exists \(0<\tilde \lambda_2<\infty\) such that if
\(a>\tilde\lambda_2\), then for all \(i,j\) in \(V_n\),
\begin{eqnarray}\label{cosh2}
\E_{\tilde \nu_V^a}\left( \left(\cosh(u^{(n)}(\delta_n,i)-u^{(n)}(\delta_n,j))\right)^8\right)\le 2.
\end{eqnarray}
Remark that in \cite{disertori2014transience}, the rooting of the field is at 0 and the graph is the restriction of the graph \(\Z^d\) to \(V_n\).
But an attentive reading of the proof shows that the result is also valid for the graph \(\ggg_n=(V_n\cup\{\delta_n\}, E_n)\)
and rooting \(\delta_n\) as well.
Indeed, the estimate is based on the protected Ward's estimates, \cite[Lemma 4]{disertori2014transience},
which remain valid for diamonds inside the set \(V_n\), and on the estimate
on effective conductances, \cite[Proposition 3]{disertori2014transience}, which is in fact an estimate inside a "diamond".
Remark that the estimate (\ref{cosh2}) is also  valid when \(i\) or \(j\) is at the boundary of the set \(V_n\)
(in fact the proof is written in the case where the diamond \(R_{i,j}\) is inside the set \(V_n\), which is the case when \(j=y_n\)
and \(i\in \Z^d\) fixed for \(n\) large enough). Specified to \(j=y_n\) and \(i\in \Z^d\) fixed, it gives for \(n\) large enough
\begin{eqnarray}\label{cosh-yn}
\E_{\tilde \nu_V^a}\left( \left(\cosh(u^{(n)}(\delta_n,i)-u^{(n)}(\delta_n,y_n))\right)^8\right)\le 2.
\end{eqnarray}
By Cauchy-Schwartz inequality, and by (\ref{cosh-voisin}) and (\ref{cosh-yn}),
we get that
\begin{eqnarray*}
\E_{\tilde \nu_V^a}\left( (\psi^{(n)}(i))^{\pm 4}\right)\le \E_{\tilde \nu_V^a}\left( e^{\pm 8u^{(n)}(\delta_n,y_n)}\right)^\demi
\E_{\tilde \nu_V^a}\left( e^{\pm 8(u^{(n)}(\delta_n,i)-u^{(n)}(\delta_n,y_n))}\right)^\demi
\le
C_{\pm}
\end{eqnarray*}
for some constant \(C_{\pm}>0\) independent of \(n\). From this we deduce by Fatou's lemma for all \(i,j\) in \(\Z^d\), \(i\sim j\),
\[
\E_{\tilde\nu_V^a} \left( \left((W_{i,j}\psi(i)\psi(j)\right)^{\pm 1}\right)
\le
\E_{\tilde\nu_V^a} \left( \left(W_{i,j}\right)^{\pm 2}\right)^\demi \E_{\tilde\nu_V^a} \left( \left(\psi(0)\right)^{\pm 4}\right)^\demi
<\infty,
\]
for \(a\) large enough.

The proof is very similar in the VRJP case, and uses Theorem 1 of \cite{disertori2010quasi}. As previously,
the estimate is valid in the case we are interested in, that is for the graph \(\ggg_n\), rooted at \(\delta_n\), and for \(x\in \Z^d\), \(y=y_n\) for
\(n\) large enough.
\end{proof}
\ali
{\bf Acknowledgment.}
We are very grateful to an anonymous referee for valuable comments.

\bibliography{bibi}
\bibliographystyle{plain}

\end{document}